\documentclass[10pt,a4paper,twoside]{amsart} 
\usepackage[english]{babel}
\usepackage[latin1]{inputenc}
\usepackage{graphicx}
\usepackage{hyperref}
\usepackage{newlfont}
\usepackage{amssymb}
\usepackage{amsmath}
\usepackage{latexsym}
\usepackage{amsthm}
\usepackage{verbatim}
\usepackage{tikz-cd}
\usepackage{enumitem}
\usepackage{afterpage}
\usepackage{xcolor,colortbl,soul}   
\usepackage{tabularx}
\counterwithin{figure}{section}
\newcolumntype{P}[1]{>{\centering\arraybackslash}p{#1}}
\usepackage{chngpage}
\usepackage{geometry}
 \usepackage{tikz}            
\usetikzlibrary{tikzmark}

\usepackage{float}

	\setlist[itemize]{leftmargin=14mm}
\theoremstyle{plain}                                 
\newtheorem{theo}{Theorem}[section]
\newtheorem{prop}[theo]{Proposition}    
       
\newtheorem{lem}[theo]{Lemma} 
\theoremstyle{definition}               
\newtheorem{defin}{Definition}
\newtheorem{exe}{Example}      
\theoremstyle{remark}        
\newtheorem{rk}{Remark}[section]          
\title{Stable cohomology of moduli spaces of hyperelliptic curves on Hirzebruch surfaces}
\author{Jonas Bergstr\"om, Angelina Zheng}
\address{Matematiska institutionen \\ Stockholms Universitet \\ 106 91 Stockholm \\ Sweden}
\email{jonasb@math.su.se}
\address{Fachbereich mathamatik \\Universit\"at T\"ubingen\\ 72076 T\"ubingen, Germany}
\email{zheng@math.uni-tuebingen.de}
\begin{document}

\begin{abstract}
We compute the stable cohomology of moduli spaces of hyperelliptic curves of a fixed genus embedded on a fixed Hirzebruch surface. We also describe these moduli spaces of embedded hyperelliptic curves in terms of moduli spaces of pointed non-embedded hyperelliptic curves.
\end{abstract}
\maketitle
\section{Introduction}
In this article we study moduli spaces of hyperelliptic curves of arbitrary genus $g$ \emph{embedded} on the $n$-th Hirzebruch surface $\mathbb{F}_n$ for arbitrary $n$. In particular we show that over the complex numbers these moduli spaces have stable rational cohomology which consists of direct sums of Tate-Hodge classes, and we find a formula for these classes in terms of the cohomology of moduli spaces of pointed curves of genus zero. 

In \cite[Proposition 1.1]{Zhe}, the second named author proved that the rational cohomology ring of the moduli space of trigonal curves of fixed genus lying on a fixed Hirzebruch surface, or equivalently the cohomology ring of a \emph{Maroni stratum}, stabilises (as the genus grows) to its tautological ring. This was done by considering trigonal curves of genus $g$ on the $n$-th Hirzebruch surface $\mathbb{F}_n$ as divisors of class $m\left[E_n\right]+d\left[F_n\right],$ with $m=3,$ $d$ depending linearly on $g,$ and $\left[E_n\right]$, $\left[F_n\right]$ the classes of the exceptional curve and of any fiber of the ruling, respectively. It was also shown in \cite[Remark 3.4]{Zhe} that, for any $m\geq4$, the cohomology of the moduli space of smooth curves embedded in $\mathbb{F}_n$ as divisors of class $m[E_n]+d[F_n],$  stabilises, as $d$ grows, to the same cohomology ring as in the case of $m=3$. 

It becomes quite natural to wonder if the same result holds for the moduli space $\mathcal{H}_{\mathbb{F}_n,g}$ of smooth curves of genus $g$ embedded in $\mathbb{F}_n$ as divisors of class $2[E_n]+(g+n+1)[F_n]$ 
(note that these are hyperelliptic curves). 
The answer to this question is negative, as we will see from the formula in Theorem~\ref{mainTheo} for its stable (as $g$ grows) cohomology. 
In particular we find that in the $m=2$ case there will be non-vanishing stable cohomology groups of unbounded degree. This is indeed different from the results both for $m=1$ (see \cite[IV.2]{ZheThesis}) and for $m\geq3$ (see \cite[Proposition~1.1]{Zhe}). Theorem~\ref{mainTheo} will be proven using the same method as in the earlier papers, namely Gorinov-Vassiliev's. Nevertheless, in this new case, the standard divisibility argument will not suffice to compute the rank of the differentials in Vassiliev's spectral sequence. A more precise description of the differentials will be made in Proposition~\ref{Diff}. This description together with an inductive argument, based on Proposition~\ref{rkDiff2}, then gives the rank. We have the impression that this type of argument could be used in other applications of Vassiliev's spectral sequence.

The moduli spaces, introduced above, of trigonal curves of genus $g$ on $\mathbb{F}_n$ with increasing $n$, induce a stratification called the Maroni stratification, of the moduli space $\mathcal T_g$ of non-embedded (abstract) trigonal curves of genus $g$ (see \cite[Section 5]{Zhe}). The situation is quite different in the case of embedded hyperelliptic curves. We will show that the moduli space of embedded hyperelliptic curves of genus $g$ on $\mathbb{F}_n$ can be described in terms of unions of certain (see Definitions~\ref{def-1}, \ref{def-2} and \ref{def-3}) moduli spaces of $l'$-pointed non-embedded hyperelliptic curves with $l'\leq g+1-n$ (see Proposition~\ref{prop-moduli} for a precise formulation). For instance, the open stratum in this description will be isomorphic to the moduli space of non-embedded hyperelliptic curves of genus $g$ together with $g+1-n$ unordered points that are not fixed by the hyperelliptic involution and such that no two of the points are interchanged by the hyperelliptic involution. 

Now, let $X$ be a complex quasi-projective variety and let $K_0(\mathsf{MHS}_{\mathbf{Q}})$ denote the Grothendieck group of mixed Hodge structures and write $\mathbf{L}$ for the class of the Tate-Hodge structure $\mathbf{Q}(-1)$ of weight $2$. 

Define the Hodge-Grothendieck polynomial of $X$ as 
$$P_X(t):=\sum_{i \geq 0}\left[H^i(X;\mathbf{Q})\right]t^i\in K_0(\mathsf{MHS}_{\mathbf{Q}})\left[t\right].$$

Say that $X$ also has an action of the symmetric group  $\mathfrak{S}_r$ that induces an action on its rational cohomology. Then we define the $\mathfrak{S}_r\mbox{-}$equivariant Hodge-Grothendieck polynomial of $X$ as
$$P_X^{\mathfrak{S}_r}(t):=\sum_{i\geq 0} \left[H^i(X;\mathbf{Q})\right]t^i\in (K_0(\mathsf{MHS}_{\mathbf{Q}})\otimes \Lambda_r)\left[t\right],$$
where $\Lambda_r$ denotes the ring of symmetric functions in $r$ variables. 
Let $\langle\cdot,\cdot\rangle$ denote the Hall inner product on $\Lambda_r$ and $s_\lambda$ the Schur polynomial corresponding to the partition $\lambda\vdash r.$

If we replace the rational cohomology of $X$ with the Borel-Moore homology, defined as the homology with locally finite support, then the corresponding Hodge-Grothendieck polynomials will be denoted by $\bar{P}_X(t)$ and $\bar{P}^{\mathfrak{S}_r}_X(t)$.

 \begin{theo}\label{mainTheo} 
The Hodge-Grothendieck polynomial $P_{\mathcal{H}_{\mathbb{F}_0,g}}(t)$ equals 
\begin{equation*}
P^{\mathrm{st}}(t):=1+\frac{\sum_{\substack{{k_1,k_2,h \geq 0} \\k_1+k_2+h\geq3}
}\mathbf{L}^{2k_1+k_2+3h}t^{3k_1+k_2+4h}\cdot\langle P_{\mathcal{M}_{0,k_1+k_2+h}}^{\mathfrak{S}_{k_1+k_2+h}}(t),s_{1^{k_1}}\cdot s_{1^{k_2}}\cdot s_h\rangle}{(1+\mathbf{L}t)(1+\mathbf{L}^2t^3)},
\end{equation*}
up to degree $\frac{g+2}{2}$. 
For $n>0,$ $P_{\mathcal{H}_{\mathbb{F}_n,g}}(t)$ equals $(1+\mathbf{L}t^2)\cdot P^{\mathrm{st}}(t)$ 
up to degree $\frac{g-n+2}{2}$. 

Furthermore, for any $n\geq 0$, all mixed Hodge structures on the rational cohomology of $\mathcal{H}_{\mathbb{F}_{n},g}$, up to degree $\frac{g-n+2}{2}$, are direct sums of Tate-Hodge structures. Hence, the Hodge-Grothendieck polynomial 
gives precisely the rational cohomology up to degree $\frac{g-n+2}{2}$. 
\end{theo}

\begin{exe}
The cohomology groups $H^i(\mathcal{H}_{\mathbb{F}_n,g};\mathbf{Q})$ such that $i \leq \frac{g-n+2}{2}$ will be called the \emph{stable} cohomology groups and $i \leq \frac{g-n+2}{2}$ will be called the stable range. 

From Theorem~\ref{mainTheo} it follows that the first $19$ stable cohomology groups of $\mathcal{H}_{\mathbb{F}_0,g}$ are 
\begin{equation*}
H^i(\mathcal{H}_{\mathbb{F}_0,g};\mathbf{Q}) \cong \begin{cases}
\mathbf{Q}&i=0,\\
0&0<i<8,\\
\mathbf{Q}(-6)&i=8,\\
\mathbf{Q}(-7)&i=9,\\
0&i=10, 11,\\
\mathbf{Q}(-9)+\mathbf{Q}(-10)&i=12,\\
\mathbf{Q}(-10)+\mathbf{Q}(-11)&i=13,\\
\mathbf{Q}(-11)&i=14,\\
\mathbf{Q}(-12)^2&i=15,\\
\mathbf{Q}(-12)^2+\mathbf{Q}(-13)^2+\mathbf{Q}(-14)&i=16,\\
\mathbf{Q}(-13)^3+\mathbf{Q}(-14)^2+\mathbf{Q}(-15)&i=17,\\
\mathbf{Q}(-14)^2+\mathbf{Q}(-15)^3&i=18.\\
\end{cases}
\end{equation*}   
\end{exe}

Let us now give an overview of the paper. 
In Section~\ref{sec-hirzebruch} we set some notation and introduce the moduli space $\mathcal{H}_{\mathbb{F}_n,g}$ of hyperelliptic curves of genus $g$ embedded 
in the Hirzebruch surface $\mathbb{F}_n$. 
In Section~\ref{sec-GV} we present Gorinov-Vassiliev's method, which we will use to compute the stable cohomology. This method is based upon an analysis of spaces of \emph{singular} curves embedded in $\mathbb{F}_n$. These give rise to strata, depending on the singularity type, whose Borel-Moore homology form the $E^1$-page of our main spectral sequence \eqref{eq-spectral} converging to the Borel-Moore homology of the space of \emph{all} singular curves. 
In Section~\ref{sec-codim} we determine the codimension of spaces of singular curves, which will be the basis of the proof of the stable range in Theorem~\ref{mainTheo}. In Sections~\ref{sec5_1}, \ref{sec-config1} and \ref{sec-config2}, we compute the Borel-Moore homology of the strata appearing in the main spectral sequence. Finally, in Section~\ref{sec-main}, we show that the main spectral sequence degenerates at its $E^1$-page, see Theorem~\ref{theo5.7}. Using Alexander duality together with the Leray-Hirsch theorem, proven in Section~\ref{sec-LH}, we can use this result to compute the cohomology of $\mathcal{H}_{\mathbb{F}_n,g}$ and finish the proof of Theorem~\ref{mainTheo}. 

Finally, in Section~\ref{sec-iso} we describe $\mathcal{H}_{\mathbb{F}_n,g}$ in terms of moduli spaces of non-embedded hyperelliptic curves of genus $g$ with unordered tuples of distinct marked points with prescribed behaviour under the hyperelliptic involution. This description also works, with small restrictions in positive odd characteristic. Therefore, it is possible to study the cohomology, via the Lefschetz fixed point theorem, in terms of counts of points over finite fields. Such a study is carried out in Section~\ref{sec-lefschetz}. These give independent results as well as consistency checks with Theorem~\ref{mainTheo}.

\section{Moduli spaces of hyperelliptic curves embedded in a Hirzebruch surface} \label{sec-hirzebruch}
For a non-negative integer $n$, let $k$ be a field of characteristic different from $2$. If $n \geq 3$ we also assume that the characteristic of $k$ does not divide $n$ and that $k$ contains $\mu_n$, the group of $n\mbox{-}$th roots of unity. 

Let $\mathbb{F}_n:=\mathbf{P}(\mathcal{O}_{\mathbf{P}^1}\oplus\mathcal{O}_{\mathbf{P}^1}(n))$ denote the $n\mbox{-}$th Hirzebruch surface defined over $k$. 
Let $\pi:\mathbb{F}_n\rightarrow \mathbf{P}^1$ be the natural projection giving $\mathbb{F}_n$ the structure of a geometrically ruled rational surface.
The set of all fibers of $\pi$ will be called the \emph{ruling} of $\mathbb{F}_n$ and the individual fibres will be called \emph{lines of the ruling}. The Picard group and intersection form of a Hirzebruch surface are well known, see \cite[V.2]{Har}, and we can identify any curve $C$ lying on it as a divisor of class
$$[C]\sim m\left[E_n\right]+d\left[F_n\right].$$
Here $F_n$ is any line of the ruling of $\mathbb{F}_n$ while $E_n$ is any line of the second ruling when $n=0$, or the unique irreducible curve having negative self-intersection when $n>0$. From now on we fix a choice of $E_0$. 

Put $m=2$. The main goal of this article is to study the stable rational cohomology of the moduli space of (smooth and proper complex) hyperelliptic curves of genus $g$ embedded in $\mathbb{F}_n$ as $g$ goes to infinity. We see that $g$ and $d$ are related, as long as $d > n$, by $g=d-n-1$. 

\begin{defin}
    For $d \geq 2n \geq 0$, let $V_{d,n}:=\Gamma(\mathcal{O}_{\mathbb{F}_n}(2[E_n]+d[F_n]))$ be the vector space of global sections of $\mathcal{O}_{\mathbb{F}_n}(2[E_n]+d[F_n])$.  Denote by $H_{d,n}\subset V_{d,n}$, the open subset defined by smooth sections, and by $\Sigma_{d,n}:=V_{d,n}\backslash H_{d,n}$ the \emph{discriminant locus}.
\end{defin}

For $d \geq 2n > 0$, using the isomorphism $\mathbb{F}_n\cong \mathrm{Bl}_{\left[0,0,1\right]} (\mathbf{P}(1,1,n)),$ we can and will identify elements in $V_{d,n}$ with polynomials in the weighted polynomial ring $R_n:=k\left[x,y,z\right]$, with $\operatorname{deg}x=\operatorname{deg}y=1$ and $\operatorname{deg}z=n$, of the form 
$$
f(x,y,z)=\alpha(x,y) z^2+\beta(x,y) z+\gamma(x,y),
$$
where $\alpha$, $\beta$, $\gamma$ are homogeneous of degree $d-2n$, $d-n$ and $d$ respectively.  
Counting parameters one sees that $v_{d,n}:=\operatorname{dim}_k(V_{d,n})=3d-3n+3$. The subset $H_{d,n}$ consists of those $f=\alpha z^2+\beta z+\gamma$ such that the discriminant $\beta^2-4\alpha \gamma$ is square-free. 

The blow-up morphism $\mathrm{Bl}_{\left[0,0,1\right]} (\mathbf{P}(1,1,n))\rightarrow \mathbf{P}(1,1,n)$ is an isomorphism over $\mathrm{Bl}_{\left[0,0,1\right]} (\mathbf{P}(1,1,n))\setminus E_n \cong \mathbf{P}(1,1,n)\setminus [0,0,1],$ and the fibre over the singular point $[0,0,1]$ is isomorphic to $E_n\cong\mathbf{P}^1$. Let $[u,v]\in \mathbf{P}^1$ be the coordinates of the exceptional divisor.
Consider now a local chart $U_z:=\{\left[x,y,z\right]\in \mathbf{P}(1,1,n);\,z\neq 0\},$ which is isomorphic, via $[x,y,z] \mapsto (x,y)$, to the affine quotient $\mathbf{A}^2/\mu_n$ of $\mathbf{A}^2$ by the action of the group $\mu_n \subset k$ of $n\mbox{-}$th roots of unity where $\zeta_n \cdot(x,y)=(\zeta_n^{-1}x,\zeta_n^{-1}y)$.  
The coordinate ring of $\mathbf{A}^2/\mu_n$ equals $k[w_0,\dots,w_{n}]\cong k[x,y]^{\mu_n},$ via $w_i\mapsto x^{n-i}y^i.$
Locally, the blow-up of the weighted projective plane can be described by $\{((x,y),\left[1,v\right]); y=xv\} \subset \mathbf{A}^2/\mu_n \times \mathbf{P}^1$ and the total transform of $f$ through the blow-up at $[0,0,1]$ is
\begin{align*}
f(x,y,1)&=\alpha(x,xv)+\beta(x,xv)+\gamma(x,xv)\\
&=x^{d-2n}\left(\alpha(1,v)+\beta(1,v)x^n+\gamma(1,v)x^{2n}\right)
\end{align*}
where $x=0$ is the equation for the exceptional curve.
Since $w_0=x^n$ in the coordinate ring of $\mathbf{A}^2/\mu_n$, the affine model of the proper transform of $f$ in $U_z$, in the affine coordinates $v,w_0$, equals
\begin{equation}\label{f_strict}
\tilde{f}(1,v,w_0)=\alpha(1,v)+\beta(1,v)w_0+\gamma(1,v)w_0^{2}.
\end{equation}

Any point $p$ lying on $E_n$ can be written in the form $p=(\left[0,0,1\right],\left[1,v\right])$, for some $\left[1,v\right]\in\mathbf{P}^1.$
Therefore $p$ will be a singular point of $f$ if and only if $\tilde{f}$ is singular at $(1,v,0)$.

Let $G_n$ be the automorphism group of $R_n$, which acts on $V_{d,n}$ via the map
\begin{equation*}
\begin{cases}
x\mapsto a_1x+b_1y,\\
y\mapsto a_2x+b_2y,\\
z\mapsto cz+\epsilon(x,y);
\end{cases}
\end{equation*}
where $a_1,a_2,b_1,b_2,c$ are constants with the property that $c(a_1b_2-a_2b_1)\neq 0$ and $\epsilon$ is a homogenous degree $n$ polynomial in the variables $x,y$.
For any $\sigma \in G_n$ and $f \in H_{d,n}$ the curves corresponding to $f$ and $\sigma f$ are isomorphic. 

For $\mathbb{F}_0\cong \mathbf{P}^1\times\mathbf{P^1}$ we can identify $V_{d,0}$ with the vector space of bi-homogeneous polynomials of bi-degree $(2,d)$ in the polynomial ring $R_0=k\left[x_0,x_1,y_0,y_1\right]$ with $\deg x_0=\deg x_1=\deg y_0=\deg y_1=1$. The automorphism group $G_0$ of $R_0$ equals $(\mathrm{GL}_2\times \mathrm{GL}_2)/\left\{aI,a^{-1}I\right\}_{a\in k^*}$ acting on the pairs of variables $x_0,x_1$ and $y_0,y_1$,  giving an induced action on $V_{d,0}$. 

\begin{defin}
For any $g \geq 2$ and $n\geq 0$ such that $g+1 \geq n$, define the moduli spaces of hyperelliptic curves of genus $g$ embedded in the Hirzebruch surface $\mathbb{F}_{n}$ to be the quotient stack $\mathcal{H}_{\mathbb{F}_n,g}:=\left[H_{g+n+1,n}/G_{n} \right]$. 
\end{defin}

Put $k=\mathbf{C},$ $G_0$ is reductive and isogenous to $\mathbf{C}^*\times \mathrm{SL}_2\times \mathrm{SL}_2$ while for $n>0$, $G_n$ is homotopy equivalent to its reductive part $\mathbf{C}^*\times \mathrm{GL}_2,$ \cite[Section 4]{Zhe}. Hence,  the stack quotient $\left[H_{g+1,0}/(\mathbf{C}^*\times SL_2\times SL_2)\right]$ has the same rational cohomology as $\mathcal{H}_{\mathbb{F}_0,g}$. For $n>0$, $\left[H_{g+n+1,n}/(\mathbf{C}^*\times \mathrm{GL}_2)\right]$ is a $\mathbf{C}^{n+1}\mbox{-}$orbifold bundle over $\mathcal{H}_{\mathbb{F}_n,g}$, which shows that they also have the same rational cohomology.

\section{Gorinov-Vassiliev's method} \label{sec-GV}
We will now compute the rational cohomology of $H_{d,n}$. This is equivalent to the Borel-Moore homology of the discriminant by Alexander duality:
\begin{equation}
\tilde{H}^\bullet(H_{d,n};\mathbf{Q})\cong\bar{H}_{2v_{d,n}-\bullet-1}(\Sigma_{d,n};\mathbf{Q})(-v_{d,n}).
\end{equation}
The Borel-Moore homology of the discriminant will be computed using Gorinov-Vassiliev's method, see \cite{TomM4}.
The main idea is to follow what has been done in the trigonal case in \cite[Section 3]{Zhe}. 
The first step is then to construct a simplicial resolution of the discriminant from a classification of singular configurations. As $d$ grows, the list of singular configurations becomes so large that it seems unmanageable to consider the complete list. 

Instead, we fix $N>1$ and compute the Borel-Moore homology of the 
strata corresponding to configurations whose index in the list is bounded by some constant {$c(N)$} that depends linearly on $N$.  
The dimension of the union of strata corresponding to all other configurations will then be bounded. 
This dimension will define the range, depending on $d,$ in which we will compute the Borel-Moore homology of the discriminant $\Sigma_{d,n}$.

Let us briefly recall the details of the method of Gorinov-Vassiliev. Given spaces $X_1,\dots,X_M$ of configurations in $\mathbb{F}_n$ satisfying the properties of \cite[List 2.3]{TomM4}, we define the cubical spaces 
$\{\mathcal{X}(I)\}_{I\subset\underline{c(N)}}$ and 
$\{\overline{\mathcal{X}}(I)\}_{I\subset\underline{c(N)}}$, over the index set 
{$\underline{c(N)}:=\{1,2,\ldots,c(N)\}$} as follows.
For
$I=\{i_1,\dots,i_r\}\subset\underline{c(N)}$ such that 
$c(N)\notin I,$
$$\mathcal{X}(I):=\{(f,\lambda_{1},\dots,\lambda_{r})\in\Sigma_{d,n}\times\prod^r_{j=1}X_{i_j}|\,\lambda_1\subset\dots\subset\lambda_r\subset\operatorname{Sing}(f)\},$$
$$\mathcal{X}({I\cup \{{c(N)}\}}):=\{(f,\lambda_{1},\dots,\lambda_{r})\in \mathcal{X}(I)|\, f\in \overline{\Sigma}_N\},\qquad \mathcal{X}(\emptyset):=\Sigma_{d,n}.$$
$$\overline{\mathcal{X}}(I):=\{(f,\lambda_{1},\dots,\lambda_{r})\in\Sigma_{d,n}\times\prod^r_{j=1}\overline{X}_{i_j}|\,\lambda_1\subset\dots\subset\lambda_r\subset\operatorname{Sing}(f)\},$$
$$\overline{\mathcal{X}}({I\cup \{{c(N)}\}}):=\{(f,\lambda_{1},\dots,\lambda_{r})\in \overline{\mathcal{X}}(I)|\, f\in \overline{\Sigma}_N\},\qquad \mathcal{X}(\emptyset):=\Sigma_{d,n}.$$
Here $\overline{\Sigma}_{N}$ is the Zariski closure of $\Sigma_N,$ the locus in $\Sigma_{d,n}$ of polynomials defining curves with at least $N$ singular points.

We also consider their geometric realisation, which is defined for $\{\mathcal{X}(I)\}_{I\subset\underline{c(N)}}$ as the quotient
$$|\mathcal{X}|=\left(\bigsqcup_{I\subset\underline{c(N)}}\mathcal{X}(I)\times \Delta(I)\right)/\sim,$$
where 
$$\Delta(I):=\left\{\tau:I\rightarrow\left[0,1\right]|\,\sum_{i\in I} \tau(i)=1\right\},$$ 
and $((f,\Lambda),\tau)\sim((f',\Lambda'),\tau')$ if and only if $(f,\Lambda)=\varphi_{IJ}(f',\Lambda')$ with $\varphi_{IJ}$ the restriction map to $I \subset J$  and $\tau'$ equals $\tau$ on $I$ and is equal to $0$ on $J \setminus I$. 
Similarly, we also define the geometric realisation of $\{\overline{\mathcal{X}}(I)\}_{I\subset\underline{c(N)}}$. 
If follows from the properties of \cite[List 2.3]{TomM4} that for any $x \in \overline{X}_i$ there is a unique $j(x)$ such that $x \in X_{j(x)}$. We then define a surjective map
$$\psi:|\overline{\mathcal{X}}|\rightarrow |{\mathcal{X}}|$$
where the class of $((f,\Lambda),\tau)\in\overline{\mathcal{X}}(I)\times\Delta(I)$ is sent to the class in $|{\mathcal{X}}|$ 
of $((f,\Xi),\upsilon)\in \mathcal{X}(J)\times\Delta(J)$ such that $J=\left\{j(\lambda_i)\in\underline{N}|\, i\in I \right\}$, $\Xi=\prod_{j\in J}\xi_j,$ with $\xi_j=\lambda_i\in X_j$ for any (arbitrary, see \cite[List 2.3]{TomM4}) choice of $i$ such that $j(\lambda_i)=j$, and $\upsilon:J\rightarrow \left[0,1\right];$ $\upsilon(j)=\sum_{i\in I | \lambda_i\in X_{j(\lambda_i)}}\tau(i).$

The map $\psi$ is compatible with the equivalence relation $\sim$ and it restricts to the identity on $|{\mathcal{X}}|\subset|\overline{\mathcal{X}}|$.  
Moreover, we give $|\overline{\mathcal{X}}|$ 
the quotient topology of the direct product topology on the $\overline{\mathcal{X}}(I)$'s while on $|{\mathcal{X}}|$ 
we consider the topology induced by $\psi$. The main purpose of introducing the space $|\overline{\mathcal{X}}|$ was indeed to give $|{\mathcal{X}}|$ this topology.

\begin{prop}{\cite[Theorem 2.8]{Gor}} The augmentation 
$|{\mathcal{X}}|\rightarrow \Sigma_{d,n}$, sending the class of $((f,\Lambda),\tau) \in \mathcal{X}(I)\times\Delta(I)$ to $f$, is a homotopy equivalence and induces an isomorphism on Borel-Moore homology groups.
\end{prop} 
Furthermore, from the theory of cubical spaces it follows that the isomorphism induced by the above augmentation is also an isomorphism of mixed Hodge structures, \cite[Section 5.8]{PS08}. Finally, the space $|{\mathcal{X}}|$ 
admits an increasing filtration 
$\mathrm{Fil}_j|{\mathcal{X}}|:=\operatorname{Im}(|{\mathcal{X}}_{\underline{j}}|\hookrightarrow |{\mathcal{X}}|),$ 
where $\mathcal{X}_{\underline{j}}$  
is the cubical space obtained by restricting $\mathcal{X}$  
to the index set $\underline{j}.$
This filtration defines a spectral sequence converging to the Borel-Moore homology of the discriminant:
\begin{equation} \label{eq-spectral}
    E^1_{p,q}\cong \bar{H}_{p+q}(F_p;\mathbf{Q})\Rightarrow\bar{H}_{p+q}(\Sigma_{d,n};\mathbf{Q}),
\end{equation}
where $F_p=\mathrm{Fil}_p\backslash \mathrm{Fil}_{p-1}.$

\section{Application of Gorinov-Vassiliev's method} \label{sec-method}
\subsection{Codimensions of spaces of singular polynomials}  \label{sec-codim}
In \cite[Section 3]{Zhe}, in connection with the trigonal case, the space $X_p$ 
of configurations was defined to be $B(\mathbb{F}_n,p)$ for $p<N$. Nevertheless, we cannot consider the same spaces here because the behaviour of the points now depends upon their position (see Example~\ref{ex}). 
More precisely, by looking at a section $f$ in $V_{d,n}$ as a polynomial, it is easy to see that spaces of polynomials singular at two distinct points will have codimension $5$ instead of $6$ precisely if the two points are on the same line of the ruling. Thus, we need to distinguish between pairs of singular points belonging to the same line of the ruling and singular points on distinct lines of the ruling. 

\begin{defin}\label{def_config}
We will say that a configuration $K\subset\mathbb{F}_n$ is of type $(k_1,k_2,h)$ if $K$ consists of: 
\begin{itemize}
\item $k_1$ distinct points on $E_n;$
\item $k_2$ distinct points on $\mathbb{F}_n\backslash E_n$, such that they all lie on distinct lines of the ruling, and each line is also distinct from the lines of the ruling containing the $k_1$ points on $E_n$; 
\item $2h$ points lying on $h$ distinct lines of the ruling, each line containing a pair of points, and each line distinct from the lines of the ruling containing the previous points.
\end{itemize} 
The space of unordered configurations of type $(k_1,k_2,h)$ will be denoted by $X_{(k_1,k_2,h)}$. 
\end{defin}

Notice that $X_{(k_1,k_2,h)} \cap X_{(k_1',k_2',h')}$ is empty if $(k_1,k_2,h) \neq (k'_1,k'_2,h')$. 
Furthermore, $(2\left[E_n\right]+d\left[F_n\right])\cdot \left[F_n\right]=2$. Hence, if a curve is to be irreducible then it cannot have more that one singular point on each ruling. If we do have more than one singular point on a line of the ruling, then the curve must be reducible with the line of the ruling as a component. 

\begin{lem}\label{codim}
Fix $k_1, k_2$ and $h$ such that $k_1+k_2+2h\geq1$. For any $n\geq 0$ and $d\geq \operatorname{max}\{k_1+k_2+5/3h+n-1,2k_1+2k_2+h+2n-1\}$, 
$$\{(f,P)\in V_{d,n}\times X_{(k_1,k_2,h)}|P
\subset\emph{Sing}(f)\}\xrightarrow{\pi} X_{(k_1,k_2,h)}$$
is a vector bundle of rank $v_{d,n}-3k_1-3k_2-5h$.
\end{lem}

\begin{proof}
It is sufficient to prove that all fibers of $\pi$ have dimension $v_{d,n}-3k_1-3k_2-5h$, which is non-negative if $d\geq k_1+k_2+5/3h+n-1$.

For simplicity, only consider the case $n\geq1$. The proof is similar to that of \cite[Lemma 2.6]{Zhe} and can be modified in a similar way to what was done there, when $n=0$.

Let us fix points $P=\{p_1,\dots,p_{k_1+k_2+2h}\}$ of a configuration of type $(k_1,k_2,h)$. 
As we already observed above, requiring $f\in V_{d,n}$ to be singular at $h$ pairs of points is equivalent to requiring its vanishing locus to contain the $h$ lines of the ruling over which the $h$ pairs of points lie.
Thus, the vector subspace of polynomials in $V_{d,n}$ which are singular at $P$ is generated by polynomials of the form
$\ell_1\cdots\ell_h\cdot f',$
where $\ell_i,\,i=1,\dots,h$, are the equations of the $h$ lines of the ruling and $f'$ is a polynomial of degree $(d-h)$ such that $f'$ vanishes at the $2h$ points and is singular at the other $k_1+k_2$ points. The $k_1+k_2$ points will be denoted by $\{q_1,\dots,q_{k_1+k_2}\}$. Counting parameters we have that the dimension of the vector subspace $W_{d,n}$ of polynomials of degree $d-h$ passing through the $2h$ points is $3(d-h)-3n+3-2h=3d-3n+3-5h=v_{d,n}-5h$. The vector space $W_{d,n}$ is then non-empty for $d\geq5/3h+n-1,$ which is always satisfied for $d\geq k_1+k_2+5/3h+n-1$.

Therefore we need to prove that, under the additional assumption $d\geq 2k_1+2k_2+h+2n-1,$ requiring $f'$ to be singular at the points $\{q_1,\dots,q_{k_1+k_2}\}$ imposes $3(k_1+k_2)$ linearly independent conditions.

Let us consider the evaluation map 
$$W_{d,n}
\xrightarrow{ev}M_{3\times (k_1+k_2)}$$ 
sending any polynomial $f'$ of the form $f'(x,y,z)=\alpha(x,y)_{d-h-2n}z^2+\beta(x,y)_{d-h-n}z+\gamma(x,y)_{d-h}$ to the $3\times (k_1+k_2)$ matrix whose $i\mbox{-}$th column is given by 
\begin{equation*}
\begin{cases}
\begin{pmatrix}\partial f'/\partial x\\ \partial f'/\partial y\\ \partial f'/\partial z\end{pmatrix}(q_i)& \text{if }q_i\not\in E_n,\\
\begin{pmatrix}\partial\alpha/\partial x\\ \partial\alpha/\partial y\\\beta\end{pmatrix}(u,v)& \text{if }q_i=(\left[0,0,1\right],\left[u,v\right])\in E_n.
\end{cases}
\end{equation*}
Recall that if $q=(\left[0,0,1\right],\left[1,v\right])\in E_n,$ the affine model $\tilde f'$ of $f'$ near $q$ is of the form \eqref{f_strict} which is singular at $q$ if and only if $(1,v,0)$ is a singular point of $\tilde f'$. An analogous result holds for $q=(\left[0,0,1\right],\left[u,1\right])\in E_n$, and hence the $i\mbox{-}$th column of the evaluation map is indeed zero precisely if $f'$ is singular at $q_i$. 

The evaluation map is linear and its kernel coincides with the fiber of $\pi$ over $P$. It is then sufficient to prove that $ev$ is surjective when $d\geq 2k_1+2k_2+h+2n-1$. Assume first that $q_1\not\in E_n$. After an appropriate change of coordinates, we may assume that $q_1=\left[1,0,0\right]$.
Consider 
$$f_0=x^r\cdot (\ell'_2\cdots\ell'_{k_1+k_2})^2 \in W_{d,n},$$
$$f_1=x^{r-1}y\cdot(\ell'_2\cdots\ell'_{k_1+k_2})^2\in W_{d,n},$$
$$f_2=x^{r-n}z\cdot(\ell'_2\cdots\ell'_{k_1+k_2})^2\in W_{d,n},$$
where $\ell'_i$ are the equations of the lines of the ruling containing $q_2,\dots,q_{k_1+k_2}$. They vanish with multiplicity $2$ at $q_2,\dots,q_{k_1+k_2},$ hence they are sent to matrices with trivial entries outside of the first column. More precisely, for $r\geq n$,
\begin{equation*}
ev(f_0)=\begin{pmatrix}r(\ell'_2\cdots\ell'_{k_1+k_2})^2(q_1)+\frac{\partial}{\partial x}(\ell'_2\cdots\ell'_{k_1+k_2})^2(q_1)&0&\dots\\ \frac{\partial}{\partial y}(\ell'_2\cdots\ell'_{k_1+k_2})^2(q_1)&0&\dots\\ 0&0&\dots\end{pmatrix},
\end{equation*}
\begin{equation*}
ev(f_1)=\begin{pmatrix}0&0&\dots\\ (\ell'_2\cdots\ell'_{k_1+k_2})^2(q_1)&0&\dots\\ 0&0&\dots\end{pmatrix},
\end{equation*}
\begin{equation*}
ev(f_2)=\begin{pmatrix}0&0&\dots\\ 0&0&\dots\\(\ell'_2\cdots\ell'_{k_1+k_2})^2)(q_1)&0&\dots\end{pmatrix}.
\end{equation*}
Since $q_1\not\in\ell'_i$ for any $i\geq2,$ $\{ev(f_0),ev(f_1),ev(f_2)\}$ is a system of linearly independent generators for the subspace in $M_{3\times ({k_1+k_2})}$ of matrices with trivial entries outside of the first column. By symmetry this can be generalized to the other columns and this proves the surjectivity of $ev$.

On the other hand, if $q_1\in E_n,$ we may assume that $q_1=(\left[0,0,1\right],\left[1,0\right])$ after an appropriate change of coordinates.
We now set
$$f_0=x^{r-2n}z^2\cdot(\ell'_2\cdots\ell'_{k_1+k_2})^2,$$
$$f_1=x^{r-2n-1}yz^2\cdot(\ell'_2\cdots\ell'_{k_1+k_2})^2,$$
$$f_2=x^{r-n}z\cdot(\ell'_2\cdots\ell'_{k_1+k_2})^2.$$
For such polynomials, if $r\geq 2n+1,$
\begin{equation*}
ev(f_0)=\begin{pmatrix}(r-2n)(\ell'_2\cdots\ell'_{k_1+k_2})^2(q_1)+\frac{\partial}{\partial x_0}(\ell'_2\cdots\ell'_{k_1+k_2})^2(p_1)&0&\dots\\ \frac{\partial}{\partial y_0}(\ell'_2\cdots\ell'_{k_1+k_2})^2(q_1)&0&\dots\\ 0&0&\dots\end{pmatrix},
\end{equation*}
\begin{equation*}
ev(f_1)=\begin{pmatrix}0&0&\dots\\ (\ell'_2\cdots\ell'_{k_1+k_2})^2(q_1)&0&\dots\\ 0&0&\dots\end{pmatrix},
\end{equation*}
\begin{equation*}
ev(f_2)=\begin{pmatrix}0&0&\dots\\ 0&0&\dots\\(\ell'_2\cdots\ell'_{k_1+k_2})^2(q_1)&0&\dots\end{pmatrix}.
\end{equation*}
Also in this case $ev(f_0),ev(f_1),ev(f_2)$ are linearly independent generators for the subspace in $M_{3\times ({k_1+k_2})}$ of matrices with trivial entries outside of the first column. 

The inequality $d\geq 2k_1+2k_2+h+2n-1$ comes from computing the 
degree of the polynomials used to prove the surjectivity of the evaluation map, 
which were of degree $d-h=r+2({k_1+k_2}-1)$. 
\end{proof}

\subsection{The configuration spaces}\label{sec5_1}
Fix $N>1$ 
for the rest of Section~\ref{sec-method}. We will now apply Gorinov-Vassiliev's method as outlined in Section~\ref{sec-GV}. 
Let $k_1,k_2,h\geq 0$ such that $k_1+k_2+2h\geq1$. Under the assumption $d\geq \operatorname{max}\{k_1+k_2+5/3h+n-1,2k_1+2k_2+h+2n-1\},$ the codimension of the vector space of polynomials singular at configurations in $X_{(k_1,k_2,h)}$ equals $c_{(k_1,k_2,h)}=3k_1+3k_2+5h$ by Lemma~\ref{codim}. 
The spaces $X_{(k_1,k_2,h)}$ are ordered by the conditions in \cite[List 2.3]{TomM4}, which coincide with the ordering by increasing codimension $c_{(k_1,k_2,h)}$, followed by increasing number of points contained in each configuration, and by possible degeneration. 
Here, a configuration $K'\in\overline{X}_{(k_1,k_2,h)}$ is a degeneration of $K\in X_{(k_1,k_2,h)}$ if there exists a Cauchy sequence $(K_j)_{j\in\mathbf{N}}$ in the space of configurations in $\mathbb{F}_n$ such that $K_0=K$ and $\lim_{j\mapsto\infty} K_j=K'$. 
Under the above assumption this ordering is equivalent to the weight order on the tuples $(k_1,k_2,h),$ first with weight vector $(3,3,5),$ then with weight vector $(1,1,2)$, followed by inverse lexicographic order. 
We will call this ordering the \emph{weighted inverse lexicographic order}. 
For a fixed $n$, we take $d$ such that \begin{displaymath}d\geq  2N+2n-1 = \operatorname{max}_{(k_1,k_2,h)\leq(N,0,0)}\{k_1+k_2+5/3h+n-1,2k_1+2k_2+h+2n-1\}.
\end{displaymath}

Our first aim is to compute the columns indexed by $j,$ such that $j< c(N)$, of the $E^1$-page in the spectral sequence \eqref{eq-spectral}, with $c(N)$ the number of triples $(k_1,k_2,h)$ such that $(k_1,k_2,h)\leq(N,0,0)$ with respect to the weighted inverse lexicographic order.
These columns consist of the Borel-Moore homology groups of $F_{(k_1,k_2,h)}$, which will be computed with the help of the following results.

\begin{prop}{\cite[Theorem 2.8, Lemmas 2.10 and 2.11]{Gor}}\label{propGor}
The stratum $F_{(k_1,k_2,h)}$ is a complex vector bundle of rank $v_{d,n}-c_{(k_1,k_2,h)}$ over a space which is a locally trivial fibration over $X_{(k_1,k_2,h)}$, whose fiber over $K\in X_{(k_1,k_2,h)}$ is a $(k_1+k_2+2h-1)\mbox{-}$dimensional open simplex which changes orientation under the homotopy class of a loop in $X_{(k_1,k_2,h)}$ interchanging a pair of points in $K$.
\end{prop}

By the above description, we will have to consider the Borel-Moore homology of $X_{(k_1,k_2,h)}$ with a local system of coefficients which is locally isomorphic to $\mathbf{Q},$ and whose monodromy representation equals the sign representation when interchanging a pair of points in $K$. 
This local system is denoted by $\pm\mathbf{Q}$ and we will call Borel-Moore homology with coefficients in this local system 
\textit{twisted Borel-Moore homology}, \cite{TomM4}. In Proposition~\ref{propConfig} we will compute 
\[
\bar{P}^{\mathrm{tw}}_{X_{(k_1,k_2,h)}}(t):=\sum_{i \geq 0}\left[H^i(X_{(k_1,k_2,h)};\pm \mathbf{Q})\right]t^i \in K_0(\mathsf{MHS}_{\mathbf{Q}})\left[t\right].
\]

The configuration spaces $X_{(k_1,k_2,h)}$ 
come with a natural projection to the quotient of $F(\mathbf{P}^1,k_1+k_2+h)$, which parametrises the lines of the ruling on which the singular points lie, by the natural action of $\mathfrak{S}_{k_1}\times\mathfrak{S}_{k_2}\times\mathfrak{S}_h$. 
More precisely, we have a fibration
\begin{equation}\label{fibr}X_{(k_1,k_2,h)}\rightarrow \mathcal{B}:=F(\mathbf{P}^1,k_1+k_2+h)/(\mathfrak{S}_{k_1}\times\mathfrak{S}_{k_2}\times\mathfrak{S}_h),\end{equation}
with fiber isomorphic to $\mathcal{F}:=\mathrm{pt}^{k_1}\times \mathbf{C}^{k_2}\times B(\mathbf{P}^1,2)^h$. 
The Leray-Serre spectral sequence, see \cite[Theorem 9.6]{DK}, associated to \eqref{fibr} is 
\begin{equation}\label{ss1}
E^2_{p,q}=\bar{H}_p(\mathcal{B};\underline{\bar{H}_q(\mathcal{F};\pm\mathbf{Q})})\Rightarrow\bar{H}_{p+q}(X_{(k_1,k_2,h)};\pm\mathbf{Q}),
\end{equation}
 where $\underline{\bar{H}_q(\mathcal{F};\pm\mathbf{Q})}$ denotes the local system of coefficients on $\mathcal{B}$ with stalk $\bar{H}_q(\mathcal{F};\pm\mathbf{Q})$, 
defined by the monodromy representation 
$$\pi_1(\mathcal{B})\rightarrow \mathrm{Aut} (\bar{H}_q(\mathcal{F};\pm\mathbf{Q}))\cong \bar{H}_q(\mathcal{F};\pm\mathbf{Q})^*,$$
where, by the K\"unneth formula, 
\begin{equation} \label{kunneth}
\bar{H}_q(\mathcal{F};\pm\mathbf{Q})=\bigoplus_{i+j+k=q} \bar{H}_i(\mathrm{pt}^{k_1};\pm\mathbf{Q})\otimes\bar{H}_j(\mathbf{C}^{k_2};\pm\mathbf{Q})\otimes\bar{H}_k(B(\mathbf{P}^1,2)^h;\pm\mathbf{Q}),\end{equation}
which is one-dimensional when $q=2k_2+2h$, and zero-dimensional otherwise.

A path in $F(\mathbf{P}^1,k_1+k_2) \subset F(\mathbf{P}^1,k_1+k_2+h)$ that interchanges a pair of points also interchanges a pair of points in the fiber $\mathrm{pt}^{k_1}\times \mathbf{C}^{k_2}$, and therefore the induced monodromy action on $\bar{H}_q(\mathcal{F};\pm\mathbf{Q})$ is by $-1$. 
If we instead consider a path in $F(\mathbf{P}^1,h) \subset F(\mathbf{P}^1,k_1+k_2+h)$ that interchanges a pair of points then in the fiber there will be two pairs of points that are interchanged. 
So, in this case the monodromy action will be trivial. 

Putting this together we have 
\begin{equation}\label{ss2}
E^2_{p,q}=\bar{H}_p(F(\mathbf{P}^1,h+k_1+k_2))^{ \mathfrak{S}_{k_1}\times\mathfrak{S}_{k_2}\times\mathfrak{S}_h}\otimes\left(\bigoplus_{\substack{i,j,k \\ i+j+k=q}} \bar{H}_i(\mathrm{pt}^{k_1};\pm\mathbf{Q})\otimes\bar{H}_j(\mathbf{C}^{k_2};\pm\mathbf{Q})\otimes\bar{H}_k(B(\mathbf{P}^1,2)^h;\pm\mathbf{Q})\right),\end{equation}
where 
$$\bar{H}_p(F(\mathbf{P}^1,h+k_1+k_2))^{ \mathfrak{S}_{k_1}\times\mathfrak{S}_{k_2}\times\mathfrak{S}_h}$$
denotes the invariant part of the Borel-Moore homology of $F(\mathbf{P}^1,h+k_1+k_2),$ with respect to the following representation of $\mathfrak{S}_{k_1+k_2+h}$: take the alternating representation on $\mathfrak{S}_{k_1}$ and $\mathfrak{S}_{k_2}$ together with the trivial representation on $\mathfrak{S}_h$, induced to $\mathfrak{S}_{k_1+k_2+h},$ via the inclusion $\mathfrak{S}_{k_1}\times\mathfrak{S}_{k_2}\times\mathfrak{S}_h\hookrightarrow\mathfrak{S}_{k_1+k_2+h}$.  

To sum up, from Proposition \ref{propGor} we can write the Borel-Moore homology of each stratum $F_{(k_1,k_2,h)}$ in terms of the twisted Borel-Moore homology of $X_{(k_1,k_2,h)}$. 
The latter will be computed using the spectral sequence \eqref{ss1} 
together with the description of its $E^2$-page in \eqref{ss2}. The twisted Borel-Moore homology of the fiber, in \eqref{ss2}, will be computed using Lemma~\ref{VasLemma} below. 
Then we will take the tensor product with the invariant part of the Borel-Moore homology of $F(\mathbf{P}^1,h+k_1+k_2),$ with respect to the induced representation described above, which can be computed with Littlewood-Richardson rule, see \cite[Appendix A.8]{FH}. 

\begin{lem}{\cite[Lemma 2]{Vas}}\label{VasLemma}
\begin{enumerate}\sloppy
\item $\bar{H}_\bullet(B(\mathbf{C}^D,k);\pm\mathbf{Q})$ is trivial for any $k\geq2$.
\item $\bar{H}_\bullet(B(\mathbf{P}^D,k);\pm\mathbf{Q})=H_{\bullet-k(k-1)}(G(k,\mathbf{C}^{D+1});\mathbf{Q}),$ 
where $G(k,\mathbf{C}^{D+1})$ is the Grassmann manifold of $k\mbox{-}$dimensional subspaces in $\mathbf{C}^{D+1}$.  
In particular $\bar{H}_\bullet(B(\mathbf{P}^D,k);\pm\mathbf{Q})$ is trivial if $k>D+1$.
\end{enumerate}
\end{lem}

\subsection{Configurations of type $(k_1,k_2,h)$ with $k_1+k_2+h\leq2$} \label{sec-config1}

The configurations of type $(k_1,k_2,h)$ with $k_1+k_2+h\leq2$ 
are listed below in the order defined in Section~\ref{sec5_1}. 
Their codimension $c_{(k_1,k_2,h)}$ is written in square brackets.

\begin{itemize}
\item[$(1,0,0):$] One point on $E_n,$ $\left[3\right]$.
\item[$(0,1,0):$] One point on $\mathbb{F}_n\backslash E_n,$ $\left[3\right]$.
\item[$(0,0,1):$] Two points on the same line of the ruling, $\left[5\right]$.
\item[$(2,0,0):$] Two points on $E_n$, $\left[6\right]$.
\item[$(1,1,0):$] Two points, one on $E_n$ and the other on $\mathbb{F}_n\backslash E$, not on the same line of the ruling, $\left[6\right]$.
\item[$(0,2,0):$] Two points on $\mathbb{F}_n\backslash E_n$, not on the same line of the ruling, $\left[6\right]$.
\item[$(1,0,1):$] Three points, two of which on the same line of the ruling, the third one on $E_n$, $\left[8\right]$.
\item[$(0,1,1):$] Three points, two of which on the same line of the ruling, the third one on $\mathbb{F}_n\backslash E_n$, $\left[8\right]$.
\item[$(0,0,2):$] Four points, two on each of two lines of the ruling, $\left[10\right]$.
\end{itemize}

For these configuration types, we will apply the methods described in Section~\ref{sec-GV} to compute the columns of the $E^1$-page of the spectral sequence \eqref{eq-spectral}.

Here we will consider the union of strata $F_{(k_1,k_2,h)}$ of configuration spaces of the same codimension and consisting of the same number of points. In these first cases the homology of the base spaces is non-trivial only in even degrees. There will therefore not be homology classes of the same weight whose degrees differ by $1,$ and hence there are no non-trivial differentials among them.

\subsubsection*{Columns $(1,0,0)$ and $(0,1,0)$} 
The spaces $F_{(1,0,0)}$ and $F_{(0,1,0)}$ are $\mathbf{C}^{v_{d,n}-3}\mbox{-}$bundles over $X_{(1,0,0)}\cong E_n$ and $X_{(0,1,0)}\cong \mathbb{F}_n\backslash E_n$ respectively, and their union is a $\mathbf{C}^{v_{d,n}-3}\mbox{-}$bundle over $\mathbb{F}_n$. The Hirzebruch surface $\mathbb{F}_n$ can be stratified into affine spaces $\mathbf{A}^0,\mathbf{A}^1, \mathbf{A}^1, \mathbf{A}^2$. Hence, its twisted Borel-Moore homology is $\mathbf{Q}$ in degree $0,$ $\mathbf{Q}(1)^2$ in degree $2,$ $\mathbf{Q}(2)$ in degree~$4$ and zero in all other degrees. Tensoring by the Borel-Moore homology of the fiber gives us the first column in~\eqref{eq-spectral}.

\subsubsection*{Column $(0,0,1)$}
The space $F_{(0,0,1)}$ is a $\mathbf{C}^{v_{d,n}-5}\mbox{-}$bundle over $X_{(0,0,1)}\xrightarrow{B(\mathbf{P}^1,2)}\mathbf{P}^1$. Here, the fiber has only one non trivial class $\mathbf{Q}(1)$ in degree $2$. Since it consists of only one factor, the local system induced on the base space must be trivial. Hence the twisted Borel-Moore homology of $X_{(0,0,1)}$ is just the tensor product of the twisted Borel-Moore homology of the fiber and of the base space, which is $\mathbf{Q}$ in degree $0,$ $\mathbf{Q}(1)$ in degree $2$ and zero in all other degrees. 

\subsubsection*{Columns $(2,0,0)$, $(1,1,0)$ and $(0,2,0)$}
The spaces $F_{(2,0,0)}$,  $F_{(1,1,0)}$ and $F_{(0,2,0)}$ are $\mathbf{C}^{v_{d,n}-6}\times\mathring{\Delta}_1\mbox{-}$bundles over $X_{(2,0,0)}\cong B(\mathbf{P}^1,2)$, $X_{(1,1,0)}\xrightarrow{\mathrm{pt}\times\mathbf{C}} F(\mathbf{P}^1,2)/(\mathfrak{S}_1\times \mathfrak{S}_1)$ and $X_{(0,2,0)}\cong B(\mathbf{P}^1,2)\times\mathbf{C}^2$, respectively. 

The Borel-Moore homology of $F_{(2,0,0)}$ and $F_{(0,2,0)}$ equals the twisted Borel-Moore homology of $X_{(2,0,0)}$ and $X_{(0,2,0)}$, respectively, tensored by the twisted Borel-Moore homology of $\mathbf{C}^{v_{d,n}-6}\times\mathring{\Delta}_1.$
In order to compute the Borel-Moore homology of $F_{(1,1,0)}$  we need to consider the invariant part, with respect to the regular representation (which is the induced representation of the trivial representation of $\mathfrak{S}_1 \times \mathfrak{S}_1$ in $\mathfrak{S}_2$), of the Borel-Moore homology of $F(\mathbf{P}^1,2)$, and finally take the tensor product with the twisted Borel-Moore of the fiber $\mathrm{pt}\times\mathbf{C}.$

\subsubsection*{Columns $(1,0,1)$ and $(0,1,1)$}
The spaces $F_{(1,0,1)}$ and $F_{(0,1,1)}$ are $\mathbf{C}^{v_{d,n}-8}\times\mathring{\Delta}_2\mbox{-}$bundles over $X_{(1,0,1)}\xrightarrow{ \mathrm{pt}\times{B(\mathbf{P}^1,2)}}F(\mathbf{P}^1,2)/(\mathfrak{S}_1\times \mathfrak{S}_1)$ and $X_{(0,1,1)}\xrightarrow{\mathbf{C}\times B(\mathbf{P}^1,2)}F(\mathbf{P}^1,2)/(\mathfrak{S}_1\times \mathfrak{S}_1)$, respectively.
The remaining computations are analogous to the ones for the case $(1,1,0).$

\subsubsection*{Column $(0,0,2)$}
The space $F_{(0,0,2)}$ is a $\mathbf{C}^{v_{d,n}-10}\times\mathring{\Delta}_3\mbox{-}$bundle over $X_{(0,0,2)}\xrightarrow{B(\mathbf{P}^1,2)^2}F(\mathbf{P}^1,2)/\mathfrak{S}_2.$
By exchanging two lines of the ruling we are exchanging two pairs of points. We will then consider the $\mathfrak{S}_2\mbox{-}$invariant part of the Borel-Moore homology of $F(\mathbf{P}^1,2)$, which is a class $\mathbf{Q} (2)$ in degree $4$. Its tensor product by $\mathbf{Q}(2)$ in degree $4$ and by $\mathbf{Q}(v_{d,n}-10)$ in degree $2v_{d,n}-17$, which are respectively the twisted Borel-Moore homology of $B(\mathbf{P}^1,2)^2$ and of $\mathbf{C}^{v_{d,n}-10}\times\mathring{\Delta}_3$, gives us the class in the last column of \eqref{eq-spectral}.\

\begin{table}[ht!]\caption{First columns, twisted by $\mathbf{Q}(-v_{d,n})$, of the spectral sequence \eqref{eq-spectral}. 
}\centering
\label{table1}
    \begin{tabular}{c|ccccc}
	$-3$&$\mathbf{Q}(-1)$ &&&&\\
	$-4$& & && &\\
	$-5$&$\mathbf{Q}(-2)^2$ & && &\\
	$-6$& & && &\\
	$-7$&$\mathbf{Q}(-3)$\tikzmark{a1} &\tikzmark{a2}$\mathbf{Q}(-3)$ && &\\
	$-8$& & &$\mathbf{Q}(-3)^2$& &  \\
	$-9$& &$\mathbf{Q}(-4)$ && & \\
	$-10$& & &$\mathbf{Q}(-4)$\tikzmark{b1}&\tikzmark{b2}$\mathbf{Q}(-4)$ &\\
	$-11$& & && & \\
	$-12$& & &$\mathbf{Q}(-5)$\tikzmark{c1}&\tikzmark{c2}$\mathbf{Q}(-5)^2$ & \\
	$-13$& & && & \\
	$-14$& & &&$\mathbf{Q}(-6)$ \tikzmark{d1}&\tikzmark{d2}$\mathbf{Q}(-6)$\\
	\hline
	&$(1,0,0)$, $(0,1,0)$&$(0,0,1)$&$(2,0,0)$, $(1,1,0)$, $(0,2,0)$&$(1,0,1)$, $(0,1,1)$&$(0,0,2)$
    \end{tabular}
    \begin{tikzpicture}[overlay, remember picture, yshift=.25\baselineskip, shorten >=.5pt, shorten <=.5pt]
		\draw [shorten >=.1cm,shorten <=.1cm,->]([yshift=2.5pt]{pic cs:a2}) -- ([yshift=2.5pt]{pic cs:a1});
	\draw [shorten >=.1cm,shorten <=.1cm,->]([yshift=2.5pt]{pic cs:b2}) -- ([yshift=2.5pt]{pic cs:b1});
	\draw [shorten >=.1cm,shorten <=.1cm,->]([yshift=2.5pt]{pic cs:c2}) -- ([yshift=2.5pt]{pic cs:c1});
	\draw [shorten >=.1cm,shorten <=.1cm,->]([yshift=2.5pt]{pic cs:d2}) -- ([yshift=2.5pt]{pic cs:d1});
    \end{tikzpicture}
\end{table}

Note that the cohomology classes of Table~\ref{table1} are divisible by the cohomology of $\mathrm{GL}_2$ (compare with Proposition~\ref{divisibility}). After computing the subsequent columns in the next subsection, we will be able to conclude that the differentials between the five columns in Table~\ref{table1} are non-trivial. In fact, from Remark~\ref{rmk-gps} below, we know that the cohomology classes in Table~\ref{table1}, together with the ones that we will compute in the following subsection, will be divisible by the cohomology of $\mathbf{C}^*\times SL_2\times SL_2$. This would not be true if the differentials in Table~\ref{table1} would be trivial. 

\subsection{Configurations of type $(k_1,k_2,h)$ with $k_1+k_2+h\geq3$} \label{sec-config2}
In this subsection we will use the method of Section~\ref{sec-GV} to compute the columns of the $E^1$-page of the spectral sequence \eqref{eq-spectral} for configurations of type $(k_1,k_2,h)$ with 
$k_1+k_2+h\geq 3$ and $(k_1,k_2,h) <(N,0,0)$. 

From \eqref{ss2}, it is sufficient to compute the Borel-Moore homology of the base space with respect to the local system whose monodromy representation in $\mathfrak{S}_{k_1+k_2+h}$ is obtained by taking the alternating representation on $\mathfrak{S}_{k_1}$ and $\mathfrak{S}_{k_2}$ together with the trivial representation on $\mathfrak{S}_h$, induced to $\mathfrak{S}_{k_1+k_2+h},$ via the inclusion $\mathfrak{S}_{k_1}\times\mathfrak{S}_{k_2}\times\mathfrak{S}_h\hookrightarrow\mathfrak{S}_{k_1+k_2+h}$.  
This can be deduced from the $\mathfrak{S}_r\mbox{-}$equivariant isomorphism 
$$\bar{H}_\bullet(F(\mathbf{P}^1,r))\cong\bar{H}_\bullet(\mathcal{M}_{0,r})\otimes\bar{H}_\bullet(\mathrm{PGL}_2),$$
see \cite{Get}.
These homology groups 
can be computed using the formula in \cite[Theorem 2.9]{KL}. 

The spectral sequence in \eqref{ss2} consists of only one row since the twisted Borel-Moore homology of the fiber is one-dimensional, hence it gives exactly the twisted Borel-Moore homology of the spaces $X_{(k_1,k_2,h)}$ for any $k_1,k_2,h$ such that $k_1+k_2+h \geq 3$. 
Using Proposition \ref{propGor}, we then determine the Borel-Moore homology of the corresponding strata $F_{(k_1,k_2,h)}$.

\begin{prop}\label{propConfig} 
For any $k_1,k_2,h$ such that $k_1+k_2+h \geq 3$, we have 
\begin{equation*}
\bar{P}^{\mathrm{tw}}_{X_{(k_1,k_2,h)}}(t)=
\mathbf{L}^{-k_2-h}t^{2k_2+2h}(\mathbf{L}^{-1}t^3+\mathbf{L}^{-3}t^6)\cdot\sum_{j=0}^k\langle\bar{P}_{\mathcal{M}_{0,k_1+k_2+h}}^{\mathfrak{S}_{k_1+k_2+h}}(t),s_{1^{k_1}}\cdot s_{1^{k_2}}\cdot s_h\rangle.
\end{equation*}
\end{prop}

\begin{prop}\label{propStrata}
For any $k_1,k_2,h$ such that $k_1+k_2+h \geq 3$, we have 
\begin{equation*}
\bar{P}_{F_{(k_1,k_2,h)}}(t)=\mathbf{L}^{-v_{d,n}+5h+3k_1+3k_2}t^{2v_{d,n}-8h-5k_1-5k_2+1}\cdot\bar{P}^{\mathrm{tw}}_{X_{(k_1,k_2,h)}}(t).
\end{equation*}
\end{prop}

\begin{rk}
From Lemma \ref{VasLemma} and the fact that Borel-Moore homology of $\mathcal{M}_{0,k_1+k_2+h}$ is minimally pure, \cite[Definitions 3.1 and 5.1]{DL} and \cite[Lemma 2.5]{KL}, the mixed Hodge structures on the twisted Borel-Moore homology of $X_{(k_1,k_2,h)},$ and the Borel-Moore homology of $F_{(k_1,k_2,h)}$ are sums of Tate-Hodge structures. Then, the polynomials in Propositions \ref{propConfig} and \ref{propStrata} give precisely the twisted Borel-Moore homology of $X_{(k_1,k_2,h)}$ and the Borel-Moore homology of $F_{(k_1,k_2,h)}$, respectively. 
\end{rk}

The following example illustrates 
the difference in behaviour compared to the trigonal case and in particular why, in the hyperelliptic case, 
we get many non-trivial contributions to our main spectral sequence.

\begin{exe}\label{ex} Let us consider configuration spaces consisting of $5$ points, namely ones of type $(1,0,2)$, $(0,1,2)$, $(2,1,1),$ $(1,2,1),$ $(4,1,0)$, $(1,4,0)$, $(5,0,0),$ and $(0,5,0)$.

Using Lemma \ref{VasLemma}, it is straightforward to see that the twisted Borel-Moore homology of $B(\mathbb{F}_n,5)$ is trivial. 

We now consider the Gysin spectral sequence converging to $\bar{H}_\bullet(B(\mathbb{F}_n,5);\pm\mathbf{Q}),$ via the stratification by $X_{(k_1,k_2,h)}$ such that $k_1+k_2+2h=5$. 
By applying Proposition \ref{propConfig} 
we get the columns of this spectral sequence, see Table~\ref{tab_ex1}.
 
\begin{table}[H]\centering \caption{Spectral sequence converging to $\bar{H}_\bullet(B(\mathbb{F}_n,5);\pm\mathbf{Q}).$}\label{tab_ex1}
\begin{tabular}{c|cccc}
10&&$\mathbf{Q}(6)$\tikzmark{ll1}&&\\
9&$\mathbf{Q}(5)$\tikzmark{ll3}&&&\tikzmark{ll}$\mathbf{Q}(6)$\\
8&&&\tikzmark{ll2}$\mathbf{Q}(5)$&\\
7&&$\mathbf{Q}(4)$\tikzmark{l1}&&\\
6&$\mathbf{Q}(3)$\tikzmark{l3}&&&\tikzmark{l}$\mathbf{Q}(4)$\\
5&&&\tikzmark{l2}$\mathbf{Q}(3)$&\\
\hline
&$(1,0,2)$&$(0,1,2)$&$(2,1,1)$&$(1,2,1)$
\end{tabular}
\begin{tikzpicture}[overlay, remember picture, yshift=.25\baselineskip, shorten >=.5pt, shorten <=.5pt]
    \draw [shorten >=.1cm,shorten <=.1cm,->]([yshift=2.5pt]{pic cs:l}) -- ([yshift=2.5pt]{pic cs:l1});
    \draw [shorten >=.1cm,shorten <=.1cm,->]([yshift=2.5pt]{pic cs:l2}) -- ([yshift=2.5pt]{pic cs:l3});
    \draw [shorten >=.1cm,shorten <=.1cm,->]([yshift=2.5pt]{pic cs:ll}) -- ([yshift=2.5pt]{pic cs:ll1});
    \draw [shorten >=.1cm,shorten <=.1cm,->]([yshift=2.5pt]{pic cs:ll2}) -- ([yshift=2.5pt]{pic cs:ll3});
\end{tikzpicture}
\end{table}
Since the twisted Borel-Moore homology of $B(\mathbb{F}_n,5)$ is trivial, the differentials among the columns in Table~\ref{tab_ex1} must be of maximal rank. 

In the trigonal case, the corresponding strata in the filtration defining the spectral sequence \eqref{eq-spectral}, 
all have the same codimension $v-15$. By Proposition~\ref{propGor} we can compute the entries of Table \ref{tab_ex2}. 
The differentials will still be of maximal rank and we get no non-trivial contributions to the Borel-Moore homology of the discriminant.

\begin{table}[h]\centering \caption{Columns from Table \ref{tab_ex1} in the spectral sequence \eqref{eq-spectral} in the trigonal case.}\label{tab_ex2}
\begin{tabular}{c|cccc}
$2v-16$&&$\mathbf{Q}(v-9)$\tikzmark{mm1}&&\\
$2v-17$&$\mathbf{Q}(v-10)$\tikzmark{mm3}&&&\tikzmark{mm}$\mathbf{Q}(v-9)$\\
$2v-18$&&&\tikzmark{mm2}$\mathbf{Q}(v-10)$&\\
$2v-19$&&$\mathbf{Q}(v-11)$\tikzmark{m1}&&\\
$2v-20$&$\mathbf{Q}(v-12)$\tikzmark{m3}&&&\tikzmark{m}$\mathbf{Q}(v-11)$\\
$2v-21$&&&\tikzmark{m2}$\mathbf{Q}(v-12)$&\\
\hline
&$(1,0,2)$&$(0,1,2)$&$(2,1,1)$&$(1,2,1)$
\end{tabular}
\begin{tikzpicture}[overlay, remember picture, yshift=.25\baselineskip, shorten >=.5pt, shorten <=.5pt]
    \draw [shorten >=.1cm,shorten <=.1cm,->]([yshift=2.5pt]{pic cs:m}) -- ([yshift=2.5pt]{pic cs:m1});
    \draw [shorten >=.1cm,shorten <=.1cm,->]([yshift=2.5pt]{pic cs:m2}) -- ([yshift=2.5pt]{pic cs:m3});
    \draw [shorten >=.1cm,shorten <=.1cm,->]([yshift=2.5pt]{pic cs:mm}) -- ([yshift=2.5pt]{pic cs:mm1});
    \draw [shorten >=.1cm,shorten <=.1cm,->]([yshift=2.5pt]{pic cs:mm2}) -- ([yshift=2.5pt]{pic cs:mm3});
\end{tikzpicture}
\end{table}

In the hyperelliptic case, 
the corresponding strata in the filtration defining the spectral sequence \eqref{eq-spectral}, 
do not have the same codimension. 
This causes not only different shifts, but also different weights depending on the configuration type. 
This gives several non-trivial contributions to the Borel-Moore homology of the discriminant, as shown in Table \ref{tab_ex3}.

\begin{table}[h]\centering \caption{Columns from Table \ref{tab_ex1} in the spectral sequence \eqref{eq-spectral}
in the hyperelliptic case.}\label{tab_ex3}
\begin{tabular}{c|cccc}
$2v-12$&&$\mathbf{Q}(v-7)$&&\\
$2v-13$&$\mathbf{Q}(v-8)$&&&\\
$2v-14$&&&&\\
$2v-15$&&$\mathbf{Q}(v-9)$&&$\mathbf{Q}(v-8)$\\
$2v-16$&$\mathbf{Q}(v-10)$&&$\mathbf{Q}(v-9)$&\\
$2v-17$&&&&\\
$2v-18$&&&&$\mathbf{Q}(v-10)$\\
$2v-19$&&&$\mathbf{Q}(v-11)$&\\
\hline
&$(1,0,2)$&$(0,1,2)$&$(2,1,1)$&$(1,2,1)$
\end{tabular}
\end{table}
\end{exe}

\subsection{The Leray-Hirsch theorem}\label{sec-LH}
The generalized version of the Leray-Hirsch theorem applies to our situation.
\begin{theo}\label{divisibility} Let 
$d>2n>0$. Fix any $f(x,y,z) \in H_{d,n}$ and define  
    \begin{align*}
	\rho_f: \mathrm{GL}_2&\rightarrow H_{d,n}\\
	\begin{pmatrix}
	   a&b\\c&d
	\end{pmatrix}&\mapsto f(ax+by,cx+dy,z).
    \end{align*}
The induced map in cohomology $\rho_f^*:H^\bullet(H_{d,n})\rightarrow H^\bullet(\mathrm{GL}_2)$ is surjective. 
It follows that there is an isomorphism of graded vector spaces with mixed Hodge structures of the form 
  \[
    H^\bullet(H_{d,n}/\mathrm{GL}_2)\otimes H^\bullet(\mathrm{GL}_2)\rightarrow H^\bullet(H_{d,n}).
  \]
\end{theo} 
\begin{proof} 
    Let us consider the natural extension $\tilde{\rho}_f:M_{2\times2}\rightarrow V_{d,n}$ and let us denote by $D$ the discriminant of $\mathrm{GL}_2$ in $M_{2\times2}.$
    The generators of $H^\bullet(\mathrm{GL}_2)$ are known to be $\left[D\right]\in\bar{H}_6(D)\cong H^1(\mathrm{GL}_2)$ and $\left[R\right]\in\bar{H}_4(D)\cong H^3(\mathrm{GL}_2),$ where $R$ is the subvariety of $D$ consisting of matrices with zeros in the first column.
		
    From Table \ref{table1} we see that $H^1(H_{d,n})\cong\bar{H}_{2v_{d,n}-2}(H_{d,n})=\left\langle\left[\Sigma_{d,n}\right]\right\rangle$ and $H^3(H_{d,n})\cong\bar{H}_{2v_{d,n}-4}(H_{d,n})=\left\langle\left[\Sigma_E\right],\left[\Sigma_F\right]\right\rangle,$ where $\Sigma_E, \Sigma_F$ denote the subspaces in $\Sigma_{d,n}$ of polynomials which are singular in a point on $E_n$ or $F_n,$ respectively.
		
    Put $A_{b,d}:=\begin{pmatrix}0&b\\0&d\end{pmatrix}\in R$. Then 
    $$\tilde{\rho}_f:A_{b,d}\mapsto C(b,d)y^{d-2n}g(y,z),$$
    where $C(b,d)\in\mathbf{C}$ and $g(y,z)$ is a homogeneous polynomial of degree $2n$. 
    Thus, since $d > 2n>0$, the image of $A_{b,d}$ is a polynomial whose zero locus is the union of the line of the ruling (defined by $y=0$) with multiplicity $d-2n$, and some other curve meeting this line.
    This means that $(\rho^*_f)^{-1}\left[R\right]=\left\langle\left[\Sigma_E\right],\left[\Sigma_F\right]\right\rangle$ and similarly one can show that the inverse image of $\left[D\right]$ is a non-zero multiple of $\left[\Sigma_{d,n}\right].$
  
    This shows that $\rho^*$ is surjective and the last statement then follows from \cite[Theorem 3]{PS03}.
\end{proof}

\begin{rk} \label{rmk-gps}
Note that the analogous construction of $\rho_f$ (to the one in Proposition~\ref{divisibility}) for the group $\mathbf{C}^*\times\mathrm{GL}_2$, does not induce a surjective map in cohomology. 
But, following \cite[Section 3.1]{TomM4}, one can show that when $n=0$ the generalized version of Leray-Hirsch theorem applies to the whole group $\mathbf{C}^*\times \mathrm{SL}_2\times \mathrm{SL}_2$. 
\end{rk}

\subsection{Main spectral sequence} \label{sec-main}
In the previous subsection we 
computed, within a stable range, the Borel-Moore homology of the strata of a filtration of $\Sigma_{d,n},$ or equivalently the first columns in the spectral sequence \eqref{eq-spectral} converging to $\bar{H}_\bullet(\Sigma_{d,n};\mathbf{Q}).$
What is left to do is to compute  
the differentials among these columns.

But first we will make some changes to the main spectral sequence. 
In order to get a simplified spectral sequence, one natural thing to do would be to group together, in the same column, strata corresponding to all configurations consisting of the same number of points.
Nevertheless, what we will do instead is to re-order the columns in the main spectral sequence and consider, in the same column, all strata corresponding to the same value of $L=k_1+k_2+h$. Notice that this does not affect the computation of the Borel-Moore homology of the discriminant since we will later prove that all differentials are trivial.
Hence, we put 
$$F_L:=\bigsqcup_{\substack{k_1,k_2,h; \\ k_1+k_2+h=L}}F_{(k_1,k_2,h)}.$$
By doing so we will consider configuration spaces, within the same column, with configurations consisting of different numbers of points. Thus, we need to keep track of the differentials that are hiding within each column.

We therefore get the following spectral sequence, 
\begin{equation}\label{ss3}
E^1_{L,q}=\bar{H}_{L+q}(F_{L}) \Rightarrow \bar{H}_{L+q}(\Sigma_{d,n}).
\end{equation}

The first columns, parametrized now by $L,$ are represented in Table \ref{Mainss_notdiv}. Notice also that in the first column we are writing only the resulting classes from the spectral sequence in Table \ref{table1}.

\begin{table}[ht!]\caption{Spectral sequence converging to $\bar{H}_\bullet(\Sigma_{d,n}),$ twisted by $\mathbf{Q}(-v_{d,n}).$}\centering
\begin{tabular}{c|cccccc}
$-3$&$\mathbf{Q}(-1)$&&&&&\\
$-4$&&&&&&\\
$-5$&$\mathbf{Q}(-2)^2$&&&&&\\
$-6$&$\mathbf{Q}(-3)^2$&&&&&\\
$-7$&&&&&&\\
$-8$&$\mathbf{Q}(-4)$&&&&&\\
$-9$&$\mathbf{Q}(-5)$&&&&&\\
$-10$&&&&&&\\
$-11$&&$\mathbf{Q}(-6)$&&&&\\
$-12$&&$\mathbf{Q}(-7)$&&&&\\
$-13$&&&$\mathbf{Q}(-7)$&&&\\
$-14$&&{$\mathbf{Q}(-8)^2$}&{$\mathbf{Q}(-8)$}&&&\\
$-15$&&$\mathbf{Q}(-9)^2$&&&&\\
$-16$&&&$\mathbf{Q}(-9)^3$&&&\\
$-17$&& $\mathbf{Q}(-10)$&$\mathbf{Q}(-10)^3$&\colorbox{red!25}{$\mathbf{Q}(-10)$}&&\\
$-18$&&$\mathbf{Q}(-11)$&& \colorbox{red!25}{$\mathbf{Q}(-10)$}$+\mathbf{Q}(-10)$&&\\
$-19$&&&$\mathbf{Q}(-11)^3$&$\mathbf{Q}(-11)^2$&$\mathbf{Q}(-11)$&\\
$-20$&&&$\mathbf{Q}(-12)^3$&$\mathbf{Q}(-12)^3$&$\mathbf{Q}(-12)$&\\
$-21$&&&&\colorbox{yellow!25}{$\mathbf{Q}(-12)^4$}$+\mathbf{Q}(-13)^2$&\colorbox{yellow!25}{$\mathbf{Q}(-12)^2$}&\\
$-22$&&&$\mathbf{Q}(-13)$&$\mathbf{Q}(-13)^2+\mathbf{Q}(-14)^4$&$\mathbf{Q}(-13)^5$&\\
$-23$&&&$\mathbf{Q}(-14)$&$\mathbf{Q}(-14)^3$&$\mathbf{Q}(-13)^3+\mathbf{Q}(-14)^4$&$\dots$\\
$-24$&&&&$\mathbf{Q}(-14)^5+\mathbf{Q}(-15)$&$\mathbf{Q}(-14)^7+\mathbf{Q}(-15)$&\\
$-25$&&&&$\mathbf{Q}(-15)^6$&$\mathbf{Q}(-15)^8$&\\
$-26$&&&&$\mathbf{Q}(-16)$&$\mathbf{Q}(-15)^6+\mathbf{Q}(-16)^6$&\\
$-27$&&&&$\mathbf{Q}(-16)^2$&$\mathbf{Q}(-16)^8+\mathbf{Q}(-17)^2$&\\
$-28$&&&&$\mathbf{Q}(-17)^2$&$\mathbf{Q}(-17)^5$&\\
$-29$&&&&&$\mathbf{Q}(-17)^3+\mathbf{Q}(-18)^4$&\\
$-30$&&&&&$\mathbf{Q}(-18)^3+\mathbf{Q}(-19)$&\\
$-31$&&&&&$\mathbf{Q}(-19)$&\\
$-32$&&&&&$\mathbf{Q}(-20)$&\\
$\dots$&&&&&&\\
\hline
&$L=1,2$&$L=3$&$L=4$&$L=5$&$L=6$&$\dots$
\end{tabular}
\label{Mainss_notdiv}
\end{table}

\begin{exe} \label{ex-nontrivial}
Consider the two classes $\mathbf{Q}(-10)$ in degrees $-13$ and $-14$ in column $L=5$ in Table \ref{Mainss_notdiv}. 
The class of degree $-13$ is defined by a configuration of type $(1,3,1)$, while the one in degree $-14$ is defined by type $(2,2,1)$. Hence, they have the same codimension and the same number of points. But the configuration of type $(2,2,1)$ can be obtained from that of type $(1,3,1)$ by degeneration: by allowing one of the points parametrized by $k_2$ to lie on the exceptional curve. This means that a non-trivial differential can exist between the two classes. Similarly, for any other pair of classes whose configurations are of type $(k_1,k_2,h)$ and $(k_1+1,k_2-1,h)$ there might be non-trivial differentials. 

Consider now the two classes $\mathbf{Q}(-12)$ of degrees $-16,-17,$ that are defined by configurations of type $(2,3,1)$, $(2,1,2)$, respectively. The latter configuration can be obtained as a degeneration of the former.
Also in this case a non-trivial differential can exist between the two classes. Similarly, for any other pair of classes whose configurations are of type $(k_1,k_2,h)$ and $(k_1,k_2-2,h+1)$ there might be non-trivial differentials.

Nevertheless, we will prove that all differentials are in fact trivial and with this we can finish our proof of Theorem \ref{mainTheo}.
\end{exe}

\begin{theo}\label{theo5.7}The spectral sequence \eqref{ss3} 
degenerates at its $E^1\mbox{-}$page. 
\end{theo}

In order to prove the theorem, we will first show that there are actually only two types of differentials that could be non-trivial (as in Example~\ref{ex-nontrivial}). Then we will show that all differentials are in fact trivial, using an inductive argument.

\begin{lem} \label{lem-pgl2}
    The rational Borel-Moore homology of $\mathrm{PGL}_2$ is 
    \begin{equation*}
    \bar{H}_i(\mathrm{PGL}_2;\mathbf Q)=\begin{cases}
    \mathbf{Q}(1)&i=3,\\
    \mathbf{Q}(3)&i=6,\\
    0&\text{otherwise.}
    \end{cases}
    \end{equation*}
\end{lem}
\begin{proof}
The cohomology, hence the Borel-Moore homology, of an affine complex algebraic group is known to be an exterior algebra with exactly one generator in odd degrees, \cite{Bor53}, \cite{Hop41}. Moreover, by regarding $\mathrm{PGL}_2$ as the topological space $\mathrm{GL}_2/\mathbf{C}^*$ one obtains that the Borel-Moore homology of $\mathrm{PGL}_2$ has two non-trivial classes: $\mathbf{Q}(1)$ in degree $3$ and $\mathbf{Q}(3)$ in degree $6$.
\end{proof}

\begin{prop}\label{Diff} Within the spectral sequence  
\eqref{ss3}, all differentials are trivial except possibly the ones of the following form: 
\begin{equation}\label{d10}
d^{\text{I}}_{(k_1,k_2,h),r}: \bar{H}_{r+1}(F_{(k_1,k_2,h)})\rightarrow\bar{H}_{r}(F_{(k_1+1,k_2-1,h)})
\end{equation}
and 
\begin{equation}\label{d11}
d^{\text{II}}_{(k_1,k_2,h),r}: \bar{H}_{r+1}(F_{(k_1,k_2,h)})\rightarrow\bar{H}_{r}(F_{(k_1,k_2-2,h+1)}).
\end{equation}
\end{prop}
\begin{proof} 

A differential in the spectral sequence \eqref{ss3} from some cohomology class $\bar{H}_{r+1}(F_{(k_1,k_2,h)})$ has, as target space, $\bar{H}_{r}(F_{(k_1',k_2',h')})$ with $(k_1',k_2',h')<(k_1,k_2,h)$. Let us recall that the order of the configuration spaces $X_{(k_1,k_2,h)}$ satisfies the conditions in \cite[List 2.3]{TomM4}. In particular $(k_1',k_2',h')<(k_1,k_2,h)$ if either any $K'\in X_{(k_1',k_2',h')}$ is such that $K'\subset K$ for $K\in X_{(k_1,k_2,h)}$ or $K'\in \overline{X}_{(k_1,k_2,h)}\setminus X_{(k_1,k_2,h)}.$

A non-trivial differential might exist if and only if there are classes of the same weight, such that the target class has degree equal to that of the class of the source space minus one. In other words, there might be a non-trivial differential if and only if the following system of equations has a solution

\begin{equation}\label{system0}
\begin{cases}
4h'+3k_1'+2k_2'-j'=4h+3k_1+2k_2-j\\
-6h'-5k_1'-3k_2'+d_j'=-6h-5k_1-3k_2-1+d_j,\\
\end{cases}
\end{equation}
where $j,j'$ denote respectively the Tate twists of classes in $\bar{H}_\bullet(\mathcal{M}_{0,k_1+k_2+h}),\bar{H}_\bullet(\mathcal{M}_{0,k'_1+k'_2+h'})$ and $d_j,d_j'$ their degrees. In the above equations we have omitted the contribution given by $\bar{H}_\bullet(\mathrm{PGL}_2),$ computed in Lemma~\ref{lem-pgl2}, which appears in both terms.  In fact, we can already conclude that any differential between two classes such that one is a tensor product of $\mathbf{Q}(1)$ and the other of $\mathbf{Q}(3)$ must be trivial.
This is because, from \cite[Section 3.1]{TomM4} and Proposition \ref{divisibility}, we know that the resulting cohomology will be divisible by that of $\mathrm{PGL}_2\subset G_n,$ which is the piece of the automorphism group acting on the exceptional curve. Thus, it is sufficient to observe that it induces an action on $F(\mathbf{P}^1,k_1+k_2+h);$ $k_1+k_2+h\geq 3$, therefore on each stratum $F_{(k_1,k_2,h)}$, and that the non-triviality of a differential between such classes would contradict the divisibility by the cohomology of $\mathrm{PGL}_2$. 

Moreover, the rational cohomology of $\mathcal{M}_{0,L}$ is minimally pure, i.e. every class of degree $l$ has weight $2l,$ see \cite[Definitions 3.1 and 5.1]{DL} and \cite[Lemma 2.5]{KL}. By duality, this means that the Borel-Moore homology class of weight $-2j$ has degree $h+k_1+k_2-3+j$. Thus $d_j=h+k_1+k_2-3+j$ and $d_j'=h'+k_1'+k_2'-3+j'$ and we can rewrite \eqref{system0} as

\begin{equation}\label{system}
\begin{cases}
4h'+3k_1'+2k_2'-j'=4h+3k_1+2k_2-j\\
-5h'-4k_1'-2k_2'-3+j'=-5h-4k_1-2k_2-4+j.\\
\end{cases}
\end{equation}

Let us consider first the case where $K'\subset K$. Then, 
\begin{itemize}
\item[a.] $(k_1',k_2',h')=(k_1,k_2,h-r)$ or
\item[b.] $(k_1',k_2',h')=(k_1-r,k_2,h)$ or
\item[c.] $(k_1',k_2',h')=(k_1,k_2-r,h)$ or 
\item[d.] $(k_1',k_2',h')=(k_1+r,k_2,h-r)$ or 
\item[e.] $(k_1',k_2',h')=(k_1,k_2+r,h-r);$
\end{itemize}
for some $r\in\mathbf{Z}_{>0}.$
By replacing $(h',k_1',k_2')$ in \eqref{system} for these first $5$ cases one can easily check that the system has no solutions, hence there are no non-trivial differentials.

Let us now assume that $K'$ is obtained from $K$ by degeneration. Then, 
\begin{itemize}
\item[f.] $(k_1',k_2',h')=(k_1+r,k_2-r,h)$ or
\item[g.] $(k_1',k_2',h')=(k_1,k_2-2r,h+r)$. 
\end{itemize}

By replacing $(k_1',k_2',h')$ in \eqref{system} we get that the corresponding systems have solutions only for $r=1$ and 
$j'-j=1$ in case f. and
$j=j'$ in case g.
We will call differentials from classes of $K$ to $K'$ of type \emph{I} if $K'\in X_{(k_1+1,k_2-1,h)}$ and of type \emph{II} if $K'\in X_{(k_1,k_2-2,h+1)}$.
\end{proof}

Notice that the differentials of type $I$ are exactly the ones hidden within each column, and those of type $II$ are between columns.

In order to prove that all differentials are trivial, we will need to construct two product maps, and consider the induced map on Borel-Moore homology.

\begin{rk}\label{rk:simplices}
    Recall the notation from Section~\ref{sec-GV}. Any point in $\operatorname{Fil}_{j}|\mathcal{X}|$ is given by a tuple $(f, \lambda_{i_{1}},\dots,\lambda_{i_r},y_I)\in \mathcal X(I)\times \Delta (I),$ with $I\subset \underline{j};  I=\{i_1,\dots,i_r; i_1<\dots<i_r\}$, and $\lambda_{i_1}\subset\dots\subset\lambda_{i_r},$ all contained in the singular locus of $f$, where we view $\lambda_{i_1},\dots,\lambda_{i_r}$ as the vertices of an $(r-1)-$dimensional simplex.
    
    Let us call $\lambda_{i_r}$ the \emph{main vertex} of the simplex, i.e. the configuration that contains all the others.
    Via the equivalence relation $\sim,$ the tuple $(\lambda_{i_{1}},\dots,\lambda_{i_r},y_I)$ is equivalent to any other tuple $(\dots,\lambda_{i_r},y_{I'})$, whose main vertex is still $\lambda_{i_r},$ and whose chain of inclusion still contains all $\{\lambda_{i_s}\}_{s=1,\dots,r}$ but which is maximal in terms of number of inclusions. Then $I'\supset I$ and the function
$y_{I'}$ is equal to $y_I$ when restricted to $I$, and zero on $I'\setminus I$. Therefore, without loss of generality, we will consider only maximal inclusions from now on.
This means that for any chain $\lambda_{i_1}\subset\dots\subset\lambda_{i_r},$ the configuration $\lambda_{i_r}\setminus\lambda_{i_{r-1}}$ consists of a single point.
Moreover, any point in $\operatorname{Fil}_{{(1,0,0)}}|\mathcal{X}|$ or $\operatorname{Fil}_{{(0,1,0)}}|\mathcal{X}|$ is defined by a simplex of dimension $0.$
\end{rk}

\begin{defin}\label{def:prod}
    Let $k_1>0$ and consider the product map $$\operatorname{Fil}_{{(1,0,0)}}|\mathcal{X}|\times\operatorname{Fil}_{{(k_1-1,k_2,h)}}|\mathcal{X}|\xrightarrow{\eta}\operatorname{Fil}_{{(k_1,k_2,h)}}|\mathcal{X}|$$
    such that, given $(f,\lambda, y_{\{(1,0,0)\}}) \in \operatorname{Fil}_{(1,0,0)}|\mathcal{X}|$ and $(g,\mu_{i_1},\dots,\mu_{i_r}, y_{I}) \in \operatorname{Fil}_{(k_1-1,k_2,h)}|\mathcal{X}|$, we put 
\begin{equation}\label{eq:eta}
\eta((f,\lambda, y_{\{(1,0,0)\}}),(g,\mu_{i_1},\dots,\mu_{i_r}, y_{I}))=
\psi(g(\mu_{i_r}\cup\lambda),\mu_{i_1},\dots,\mu_{i_r},\mu_{i_r}\cup\lambda,y_{I'}).
\end{equation}
Let $k$ be the unique integer such that $\mu_{i_r}\cup\lambda\in X_k$. Then $g(\mu_{i_r}\cup\lambda)$ is the polynomial obtained by taking the orthogonal projection of $g$ in $L_k(\mu_{i_r}\cup\lambda),$ the image of $\mu_{i_r}\cup\lambda\in X_k$ via the continuous map in \cite[List 2.3 (5)]{TomM4}: it is the linear subspace of polynomials singular at the configuration $\mu_{i_r}\cup\lambda$ 
and we have $g(\mu_{i_r}\cup\lambda)=g$ if $\lambda\subset\operatorname{Sing}(g).$
The map $\psi$ is the continuous surjective map
$\psi:|\overline{\mathcal{X}}|\rightarrow{|\mathcal{X}|},$ contracting all closed simplices corresponding to inclusions of configurations that are equalities. 
The index set equals 
$$I':=I\cup\{k\in \{(k_1-1,k_2,h),\dots,(k_1,k_2,k_3)\};\mu_{i_r}\cup\lambda\in X_k\}$$
and we let $i_1,\dots,i_r,i_{r+1}$ 
denote its elements. Note that $i_{r+1}$ appears only if $\lambda\not\subset\operatorname{Sing}(g)$ or $\lambda\not\subset\mu_{i_r}.$
Put $\delta_{k}^{I'}=1$ if $k\in I'$ and zero otherwise. The function $y_{I'}$ is 
\begin{align}\label{def:y_I}
y_{I'}(a)=\begin{cases}
y_{I}(a)&a\in I\setminus\{i_r\}\\
\frac{y_I(i_r)^2}{y_I(i_r)+\delta_{i_{r+1}}^{I'}}&a=i_{r}\\
\frac{y_I(i_r)}{y_I(i_r)+1}&a=i_{r+1}\\
\end{cases}
\end{align}
when it contains the maximum number of indices, otherwise just ignore the lines related to indices not in $I'$.

For $k_2>0,$ we can, in a similar way, define a product map 
   $$\operatorname{Fil}_{{(0,1,0)}}|\mathcal{X}|\times\operatorname{Fil}_{{(k_1,k_2-1,h)}}|\mathcal{X}|\xrightarrow{\eta'}\operatorname{Fil}_{{(k_1,k_2,h)}}|\mathcal{X}|.$$

\end{defin}

\begin{lem}\label{lm:rkDiff2_1}
The product maps $\eta,\eta'$ in Definition \ref{def:prod} are continuous and proper.
\end{lem}
\begin{proof}
The spaces $\operatorname{Fil}_j|\mathcal{X}|$, $\forall j\subset\underline{c(N)},$ are all compact and Hausdorff, \cite[Section 2, Proposition 2.7]{Gor} and from \eqref{eq:eta} one sees that $\eta,\eta'$ are continuous, hence proper.
\end{proof}

\begin{prop}\label{rkDiff2} 
Let $j,k_1,k_2,h$ be positive integer with $(k_1,k_2)\neq (0,0),$  $k_1+k_2+h\geq 3.$ 
If $d^{{I}}_{(k'_1,k'_2,h),j'}$ is trivial for any $j',k_1',k_2'$ with  $k_1'+k_2'+h\geq 3,$ $j'\leq j$,   $k_1'\leq k_1$ and $k_2'\leq k_2$ such that either $k_1'<k_1$ or $k_2'<k_2,$
then $d^{{I}}_{(k_1,k_2,h),j}$ is trivial.
The same statement also holds with $d^{{I}}$ replaced by $d^{{II}}$. 
\end{prop}
\begin{proof}
Let us consider the product maps $\eta,\eta'$ defined in Definition \ref{def:prod}. By being continuous and proper, they 
induce homomorphisms in Borel-Moore homology and we find that 
\[
\begin{tikzcd}
\bar{H}_{j}(\operatorname{Fil}_{{(1,0,0)}}|\mathcal{X}|\times \operatorname{Fil}_{{(k_1-1,k_2,h)}}|\mathcal{X}|)\arrow{d}{\eta_{j}}\arrow{r}{\sim}&\bigoplus_{p+q=j}\bar{H}_p(\operatorname{Fil}_{{(0,0,1)}}|\mathcal{X}|)\otimes\bar{H}_q(\operatorname{Fil}_{{(k_1-1,k_2,h)}}|\mathcal{X}|) \\
\bar{H}_{j}(\operatorname{Fil}_{{(k_1,k_2,h)}}|\mathcal{X}|)
\end{tikzcd}
\]
and
\[
\begin{tikzcd}
\bar{H}_{j}(\operatorname{Fil}_{{(0,1,0)}}|\mathcal{X}|\times \operatorname{Fil}_{{(k_1,k_2-1,h)}}|\mathcal{X}|)\arrow{d}{\eta'_{j}}\arrow{r}{\sim}&\bigoplus_{p+q=j}\bar{H}_p(\operatorname{Fil}_{{(0,1,0)}}|\mathcal{X}|)\otimes\bar{H}_q(\operatorname{Fil}_{{(k_1,k_2-1,h)}}|\mathcal{X}|) \\
\bar{H}_{j}(\operatorname{Fil}_{{(k_1,k_2,h)}}|\mathcal{X}|).
\end{tikzcd}, 
\]
where $\eta_j$, $\eta'_j$ denotes the morphisms induced by $\eta$, $\eta'$ in Borel-Moore homology. The maps $\eta_j$, $\eta'_j$ commute with the differential $\partial$ over $\operatorname{Fil}_k|\mathcal{X}|,$ which, by definition of a spectral sequence of a filtration, \cite[5.4]{Wei}, induces the differential $d$ in the columns of the spectral sequence.

Let $\xi\in\bar{H}_{{j}}(F_{(k_1,k_2,h)}),$ then 
$$d_{{(k_1,k_2,h),j}}^I(\xi)=\left[\partial x\right],$$ where $x$ denotes an element in the complex of locally finite singular chains in $\operatorname{Fil}_{(k_1,k_2,h)}|\mathcal{X}|$.
It is not hard to see, from Remark~\ref{rk:simplices}, 
that for any continuous map $\sigma:\Delta_I\to \operatorname{Fil}_{(k_1,k_2,h)}|\mathcal{X}|$, or equivalently any simplex in $\operatorname{Fil}_{(k_1,k_2,h)}|\mathcal{X}|$ defined over the index set $I$, there exist two simplices $\sigma_1,\sigma_2$ with $\sigma_1\in \operatorname{Fil}_{(1,0,0)}|\mathcal{X}|\cup\operatorname{Fil}_{(0,1,0)}|\mathcal{X}|$ and $\sigma_2\in\operatorname{Fil}_{(k_1-1,k_2,h)}|\mathcal{X}|\cup\operatorname{Fil}_{(k_1,k_2-1,h)}|\mathcal{X}|$ such that $\eta(\sigma_1,\sigma_2)=\sigma$ or $\eta'(\sigma_1,\sigma_2)=\sigma$. All points within the same simplex are indeed defined by the same main vertex $\mu$. The choice between $\eta, \eta'$ depends on the maximal proper configuration $\mu'$ contained in the main vertex, and the pair is given by the $0\mbox{-}$simplex $\mu\setminus\mu'$ and the simplex obtained by removing the main vertex.

Therefore, given $x$ there exist $(a, b),(c,d)$ locally finite chains in $\operatorname{Fil}_{(1,0,0)}|\mathcal{X}|\times\operatorname{Fil}_{(k_1-1,k_2,h)}|\mathcal{X}|$ and \sloppy$\operatorname{Fil}_{(0,1,0)}|\mathcal{X}|\times\operatorname{Fil}_{(k_1,k_2-1,h)}|\mathcal{X}|$ respectively such that $x=\eta(a,b)+\eta'(c,d),$ where $\eta$ and $\eta'$ now 
denote the induced maps on locally finite chains. It follows that 
\begin{align*}
d_{{(k_1,k_2,h),j}}^I(\xi)&=\left[\partial\eta(a,b)+\partial\eta'(c,d)\right]\\
&=\eta_j(\left[\partial(a,b)\right])+\eta'_j(\left[\partial(c,d)\right])\\
&=\eta_{j}\left(\sum_{p+q=j}d_{(1,0,0),p}^I(\alpha_p)\otimes\beta_q\oplus(-1)^p\alpha_p\otimes d_{(k_1-1,k_2,h),q}^I(\beta_q)\right)\\
&+\eta'_{j}\left(\sum_{p+q=j}d_{(0,1,0),p}^I(\gamma_p)\otimes\delta_q\oplus(-1)^p\gamma_p\otimes d_{(k_1,k_2-1,h),q}^I(\delta_q)\right),
\end{align*}
where $\alpha_p\in \bar{H}_p(F_{(1,0,0)}),\beta_q\in\bar{H}_q(F_{(k_1-1,k_2,h)}),\gamma_p\in \bar{H}_p(F_{(0,1,0)}),\delta_q\in\bar{H}_q(F_{(k_1,k_2-1,h)})$, and where the last equality comes from the K\"unneth formula.
By assumption, together with the fact that the differentials from cohomology classes of $F_{(1,0,0)}$ and $F_{(0,1,0)}$ are trivial from Table \ref{table1}, we have that the right hand is trivial. Repeating for any non trivial classes in 
$\bar{H}_j(F_{(k_1,k_2,h)})$, yields the proof for $d^I$.
The same will hold by considering instead the differential $d^{II}.$
\end{proof}

\begin{proof}[Proof of Theorem \ref{theo5.7}] 
   Observe first from Table \ref{Mainss_notdiv} that the column $F_L$, with $L=3,$ has trivial differentials.  
    From Proposition \ref{Diff} it follows that  
    for any $h\geq 3$ and $j\in \mathbf{Z}$, $d^I_{(0,0,h),j}$, $d^{II}_{(0,0,h),j}$ are 
    trivial.
    Indeed, since $k_1,k_2$ are both zero in $(0,0,h),$ the target space of both differentials would be zero as well. 
    Then, by Proposition~\ref{rkDiff2}, 
    any other differential of the spectral sequence \eqref{ss3} with target space $\bar{H}_j(F_{(k_1',k_2',h)})$ with  $k_1'+k_2'+h=3$ is trivial. Hence the differentials $d^I_{(k_1,k_2,h),j}$, $d^{II}_{(k_1,k_2,h),j}$ are trivial for any $j$ and $k_1,k_2,h$ such that $k_1+k_2+h=3, 4$. The result now follows from Proposition~\ref{rkDiff2} and induction. 
\end{proof}

\begin{lem}\label{lm:Leray}
The Leray spectral sequence associated to 
$$H_{d,n}/\mathrm{GL}_2\xrightarrow{\mathbf{C}^*}H_{d,n}/(\mathbf{C}^*\times \mathrm{GL}_2),$$
is represented in Table \ref{Lerayss}, with the highlighted differentials of rank $1$.
\vspace{-0.5cm}
\begin{table}[H]\caption{Leray spectral sequence of $H_{d,n}/\mathrm{GL}_2\xrightarrow{\mathbf{C}^*}H_{d,n}/(\mathbf{C}^*\times \mathrm{GL}_2).$}
\centering
\begin{tabular}{c|cccccccccc}
$1$&$\mathbf{Q}(-1)$\tikzmark{a}&&$\mathbf{Q}(-2)$&&$\mathbf{Q}(-7)$\tikzmark{c}&$\mathbf{Q}(-8)$\tikzmark{e}&$\mathbf{Q}(-8)$&$\mathbf{Q}(-9)$&$\mathbf{Q}(-10)+\mathbf{Q}(-11)$&\\
$0$&$\mathbf{Q}$&&\tikzmark{b}$\mathbf{Q}(-1)$&&$\mathbf{Q}(-6)$&$\mathbf{Q}(-7)$&\tikzmark{d}$\mathbf{Q}(-7)$&\tikzmark{f}$\mathbf{Q}(-8)$&$\mathbf{Q}(-9)+\mathbf{Q}(-10)$&\\
\hline
&$0$&$1$&$2$&$\dots$&$8$&$9$&$10$&$11$&$12$&$\dots$
\label{Lerayss}
\end{tabular}
\begin{tikzpicture}[overlay, remember picture, yshift=.25\baselineskip, shorten >=.5pt, shorten <=.5pt]
		\draw [shorten >=.1cm,shorten <=.1cm,->]([yshift=5pt]{pic cs:a}) -- ([yshift=5pt]{pic cs:b});
		\draw [shorten >=.1cm,shorten <=.1cm,->]([yshift=5pt]{pic cs:c}) -- ([yshift=5pt]{pic cs:d});
		\draw [shorten >=.1cm,shorten <=.1cm,->]([yshift=5pt]{pic cs:e}) -- ([yshift=5pt]{pic cs:f});
		\end{tikzpicture}
\end{table}
\end{lem}

\begin{proof}
Since the cohomology of $\mathbf{C}^*$ is $\mathbf{Q}$ in degree $0,$ $\mathbf{Q}(-1)$ in degree $1$ and zero in all other degrees, the Leray spectral sequence will have two rows, related by a Tate twist.

The first columns of the spectral sequence will look either as the ones in Table \ref{LerayssCol}, or as those in Table \ref{LerayssAlt}, where all highlighted differentials have rank one.

\begin{table}[H]\caption{Possibility 1 for the first seven columns of the Leray spectral sequence of $H_{d,n}/\mathrm{GL}_2\xrightarrow{\mathbf{C}^*}H_{d,n}/(\mathbf{C}^*\times \mathrm{GL}_2).$}\centering
\begin{tabular}{c|cccccc}
$1$&$\mathbf{Q}(-1)$\tikzmark{1}&&$\mathbf{Q}(-2)$&&&\\
$0$&$\mathbf{Q}$&&\tikzmark{2}$\mathbf{Q}(-1)$&&&\\
\hline
&$0$&$1$&$2$&$3$&$\dots$&
\end{tabular}
\begin{tikzpicture}[overlay, remember picture, yshift=.25\baselineskip, shorten >=.5pt, shorten <=.5pt]
		\draw [shorten >=.1cm,shorten <=.1cm,->]([yshift=5pt]{pic cs:1}) -- ([yshift=5pt]{pic cs:2});
		\end{tikzpicture}
\label{LerayssCol}
\end{table}
 
\begin{table}[H]\caption{Possibility 2 for the first seven columns of the Leray spectral sequence of $H_{d,n}/\mathrm{GL}_2\xrightarrow{\mathbf{C}^*}H_{d,n}/(\mathbf{C}^*\times \mathrm{GL}_2).$}\centering
\begin{tabular}{c|P{1cm} P{1cm} P{1cm} P{1cm} P{1cm} P{1cm} P{1cm} }
$1$&$\mathbf{Q}(-1)$\tikzmark{3}&&$\mathbf{Q}(-2)$\tikzmark{5}&$\mathbf{Q}(-3)$\tikzmark{7}&$\mathbf{Q}(-3)$&$\mathbf{Q}(-4)$&\\
$0$&$\mathbf{Q}$&&\tikzmark{4}$\mathbf{Q}(-1)$&$\mathbf{Q}(-2)$&\tikzmark{6}$\mathbf{Q}(-2)$&\tikzmark{8}$\mathbf{Q}(-3)$&\\
\hline
&$0$&$1$&$2$&$3$&$4$&$5$&$\dots$
\end{tabular}
\begin{tikzpicture}[overlay, remember picture, yshift=.25\baselineskip, shorten >=.5pt, shorten <=.5pt]
		\draw [shorten >=.1cm,shorten <=.1cm,->]([yshift=5pt]{pic cs:3}) -- ([yshift=5pt]{pic cs:4});
		\draw [shorten >=.1cm,shorten <=.1cm,->]([yshift=5pt]{pic cs:5}) -- ([yshift=5pt]{pic cs:6});
		\draw [shorten >=.1cm,shorten <=.1cm,->]([yshift=5pt]{pic cs:7}) -- ([yshift=5pt]{pic cs:8});
		\end{tikzpicture}\label{LerayssAlt}
\end{table}

The correct spectral sequence is the one represented in Table \ref{LerayssCol}. In fact, the one represented in Table \ref{LerayssAlt}, would give two extra classes $\mathbf{Q}(-r)$ in degree $2r-1$ and $\mathbf{Q}(-s)$ in degree $2s-2$ for some $r,s>2$. It is then sufficient to check that classes satisfying such relations between weight and degree cannot exist. In fact, for instance, the class $\mathbf{Q}(-r),$ by Alexander duality, is a class $\mathbf{Q}(v_{d,n}-r),$ of degree $2v_{d,n}-2r$ in Borel-Moore homology. From the construction of the spectral sequence, the existence of this class would imply the existence of a tuple $(h,k_1,k_2); h+k_1+k_2\geq 3$ whose corresponding stratum gives a class whose weight and degree are such that 
\begin{equation*}
\begin{cases}
    r=4h+3k_1+2k_2-1-j\\
    2r=5h+4k_1+2k_2-1-j
\end{cases}
\end{equation*}
with $0\leq j\leq h+k_1+k_2-3$. The above system gives $r=h+k_1$ as a solution. But from the first equation and the constraint on $j$, we can easily see that $r\geq 3h+2k_1+k_2+2>h+k_1,$ which is a contradiction. 

The Leray spectral sequence $H_{d,n}/\mathrm{GL}_2\xrightarrow{\mathbf{C}^*}H_{d,n}/(\mathbf{C}^*\times \mathrm{GL}_2)$ 
is then given in Table~\ref{Lerayss},  
with the highlighted differentials of rank $1$ being the only possible choice.
In fact, from the description of the first seven columns of the Leray spectral sequence in Table~\ref{LerayssCol}, we have that $H^2(H_{d,n}/(\mathbf{C}^*\times \mathrm{GL}_2))$ is a non-zero multiple of the Euler class $\xi$ of the $\mathbf{C}^*\mbox{-}$bundle, and $\xi^2=0$ in $H^4(H_{d,n}/(\mathbf{C}^*\times \mathrm{GL}_2))$. Any other set of non-trivial differentials 
from the one represented in Table \ref{Lerayss} would give a non trivial class in $H^{2j}(H_{d,n}/(\mathbf{C}^*\times \mathrm{GL}_2))$ which would be a non-zero multiple of $\zeta\cdot\xi^2,$ for some $\zeta$ generating $H^{2j-4}(H_{d,n}/(\mathbf{C}^*\times \mathrm{GL}_2))$.  This leads to a contradiction since necessarily $\xi^2=0$. 
\end{proof}

\begin{proof}[Proof of Theorem~\ref{mainTheo}]

We know from \cite[Section 3.1]{TomM4} that a generalized version of the Leray-Hirsch theorem can be applied for $n=0$ 
and from Theorem \ref{theo5.7} we know that all differentials in the spectral sequence converging to $H^\bullet(H_{d,0})$ are trivial. Hence the cohomology of $H^\bullet(H_{d,0}/(\mathbf{C}^*\times SL_2\times SL_2))$ can be read from Table \ref{Mainss_notdiv}, after dividing by  $H^\bullet(\mathbf{C}^*\times SL_2\times SL_2)$. 

On the other hand, in order to get the rational cohomology of $\left[H_{d,n}/(\mathbf{C}^*\times \mathrm{GL}_2)\right]$, we should consider first that of $H_{d,n}/\mathrm{GL}_2$. The latter can be recovered by taking the tensor product of the cohomology of $H_{d,0}/(\mathbf{C}^*\times SL_2\times SL_2)$ and that of $SL_2$. In fact, as a consequence of Proposition \ref{divisibility}, the generalised version of Leray-Hirsch theorem can be applied to $\mathrm{GL}_2$ when $n>0$. 

Then the rational cohomology of $\left[H_{d,n}/(\mathbf{C}^*\times \mathrm{GL}_2)\right]$ is given by Lemma~\ref{lm:Leray}.

The mixed Hodge structures appearing in the spectral sequence are sums of Tate-Hodge structures and taking the cohomology of the base space yields the description in Theorem \ref{mainTheo}.

Finally, we compute the stable range in Theorem~\ref{mainTheo}. This is obtained by a similar argument in \cite[Section 3]{Zhe}. 
Define $$F_{\geq N}=\bigsqcup_{L\geq N} F_L.$$ 
Intuitively, we would like $F_{\geq N}$ to be the union of the stratum $F_{(N,0,0)}$ and all the other strata with higher index, with respect to the ordering defined in Section \ref{sec5_1}, which is defined by the conditions in \cite[List 2.3]{TomM4}. Nevertheless, the above definition of $F_{\geq N}$ does not include some strata, for instance the stratum $(0,0,N-1)$ which is defined by configurations of $2N-2$ points but $(N,0,0)<(0,0,N-1).$
Let us now explain why it is enough to consider $F_{\geq N}$ instead of $$F_{\geq (N,0,0)}:=\bigsqcup_{(k_1,k_2,h)\geq (N,0,0)}F_{(k_1,k_2,h)}.$$
The space $F_{\geq (N,0,0)}\setminus F_{\geq N}$ consists of strata 
$F_{(k_1,k_2,h)}$
such that $(k_1,k_2,h)\geq(N,0,0)$ and $L=k_1+k_2+h<N.$
Recall that $\bar{H}_\bullet(\mathcal{M}_{0,L})$ is non-trivial in degree less than or equal to $2(L-3),$ and the inequality $(k_1,k_2,h)\geq(N,0,0)$ implies that $c_{(k_1,k_2,k_3)}\geq 3N$. 
By Proposition \ref{propStrata}, any stratum $F_{(k_1,k_2,h)}$
would give non-trivial contribution in the main spectral sequence \eqref{eq-spectral} in degree less or equal than $2v_{d,n}-6h-5k_1-3k_2-1+2L\leq 2v_{d,n}-3N+2L-1\leq 2v_{d,n}-N-3$. By Alexander duality, any such stratum would then give non-trivial contribution, in cohomology, in degree $i\geq N+2$ which is outside the stable range that we will now compute. 

Recall that $\mathcal{X}(\{c(N)\})=\overline{\Sigma}_{N},$ the Zariski closure of the locus $\Sigma_{N}\subset\Sigma_{d,n}$ of polynomials which are singular at least at $N$ singular points. 
We can stratify the space $F_{\geq N}$ as the union of 
$$\Phi_0=\left(\mathcal{X}(\{c(N)\})\times\Delta(\{c(N)\})\right)/\sim,$$ 
and
$$\Phi_{(j_1,j_2,j_3)}=\left(\bigsqcup_{\operatorname{max}J=(j_1,j_2,j_3)}\mathcal{X}(J\cup\{c(N)\})\times\Delta(J\cup\{c(N)\})\right)/\sim;\qquad (j_1,j_2,j_3)< (N,0,0).$$
Using an argument similar to that in the proof of Proposition \ref{propGor}, any $\Phi_{(j_1,j_2,j_3)}$ is a locally trivial fibration over $\mathcal{X}(\{(j_1,j_2,j_3),c(N)\})$ with fiber isomorphic to the interior of a $j\mbox{-}$dimensional simplex, with $j=j_1+j_2+2j_3.$
We will then have a surjective map 
\begin{multline*}
    {\mathcal{Y}}^{\operatorname{ord}}_{(k_1,k_2,h)}:=\{(f,{K}^{\operatorname{ord}})\in V_{d,n}\times {X}^{\operatorname{ord}}_{(k_1,k_2,h)}|{K}^{\operatorname{ord}}\subset\operatorname{Sing}(f)\}\rightarrow \mathcal{X}(\{(j_1,j_2,j_3),c(N)\})\\
    (f,p_1,\dots,p_{k_1},p_1',\dots,p'_{k_2},\{q_1,q'_1\},\dots,\{q_h,q_h'\}) \mapsto (f,\{p_1,\dots,p_{j_1},p_1',\dots,p'_{j_2},\{q_1,q'_1\},\dots,\{q_{j_3},q_{j_3}'\}\});
\end{multline*}
where $k_1+k_2+h=N,$ $h=j_3,$ and ${X}^{\operatorname{ord}}_{(k_1,k_2,h)}$ is the space of ordered configurations of type $(k_1,k_2,h).$
By an analogous argument to the one proving Lemma \ref{codim}, the space ${\mathcal{Y}}^{\operatorname{ord}}_{(k_1,k_2,h)}$ is a vector bundle or rank $v-3(k_1+k_2)-5h$ over $X_{(k_1,k_2,h)}^{\operatorname{ord}}$ provided that $d\geq2N+2n-1$ holds.
The complex dimension of $X_{(k_1,k_2,j_3)}^{\operatorname{ord}}$ is $k_1+2k_2+3j_3,$ and thus $\operatorname{dim}_{\mathbf{R}}\mathcal{X}(\{(j_1,j_2,j_3),N\})\leq 2v-4k_1-2k_2-4j_3$. The spaces $\Phi_{(j_1,j_2,j_3)}$ therefore have real dimension at most $2v-4k_1-2k_2-2j_3+j_1+j_2$ and the largest among them is $\Phi_{(0,N-1,0)},$ which has real dimension at most $2v-N-1$. 
Then, by Alexander duality, the space $F_{\geq N}$ gives no contribution in cohomology in degree $i<N$. By the inequality $d\geq2N+2n-1$ we can conclude that the stable range is $i<{(d-2n+1)}/{2}=(g-n+2)/2.$
\end{proof}

\section{Moduli spaces of pointed non-embedded hyperelliptic curves}  
\begin{subsection}{An isomorphism of moduli spaces} \label{sec-iso} 
We will now relate $\mathcal{H}_{\mathbb{F}_n,g}$ over a field $k$, to a moduli space of pointed hyperelliptic curves. We make the same assumptions on $k$ as in the beginning of Section~\ref{sec-hirzebruch}.

For $g \geq 2$, $l \geq 0$ such that $g \geq l$ we identify $H_{2g+2-l,g+1-l}$ with $D_{g,l}$ consisting of the set of polynomials $\alpha$, $\beta$, $\gamma$ in variables $x,y$, homogeneous of degree $l$, $g+1$ and $2g+2-l$ respectively and such that $\beta^2-4\alpha\gamma$ is square-free. 
The identification gives an induced action of $G_{g+1-l}$ on $D_{g,l}$, and $\mathcal{H}_{\mathbb{F}_{g+1-l},g}$ is isomorphic to the stack quotient $\mathcal D_{g,l}:=[D_{g,l}/G_{g+1-l}]$. 

The curve given by $\alpha z^2+\beta z+\gamma=0$ on $\mathbb{F}_{g+1-l}$ for any $(\alpha,\beta,\gamma) \in D_{g,l}$ is, by the change of variables 
$z \mapsto (z-\beta)/(2\alpha)$, isomorphic to the curve given by the equation $z^2-(\beta^2-4\alpha \gamma)=0$. 
\begin{defin}
For $g \geq 2$, $l \geq 0$ such that $g \geq l$, let $D'_{g,l}$ consist of polynomials $\alpha$, $\beta$, $\delta$ in variables $x,y$, homogeneous of degree $l$, $g+1$ and $2g+2$ respectively, and such that $\delta$ is square-free and $\alpha$ divides $\beta^2-\delta$.    
\end{defin}

Define a map $\Psi: D_{g,l} \to D'_{g,l}$ by $(\alpha,\beta,\gamma) \mapsto (\alpha,\beta,\beta^2-4\alpha\gamma)$. This map is clearly an isomorphism with inverse $(\alpha,\beta,\delta) \mapsto (\alpha,\beta,(\beta^2-\delta)/(4\alpha))$. Let $G_{g+1-l}$ act on $D'_{g,l}$ via $\Psi$, i.e. $ \sigma f=\Psi(\sigma \Psi^{-1}(f))$ for any $\sigma \in G_{g+1-l}$ and $f \in D'_{g,l}$. We then get an isomorphism, also denoted $\Psi$, between the stack quotients $\mathcal D_{g,l}$ and $\mathcal D'_{g,l}:=[D'_{g,l}/G_{g+1-l}]$. 

\begin{defin} \label{def-1}
For $g \geq 2$ and $l \geq 0$, let $\mathcal H_{g,l}$ denote the moduli space (stack) of smooth hyperelliptic curves of genus $g$ together with $l$ marked (distinct) points. There is an action of the symmetric group $\mathfrak{S}_l$, permuting the $l$ marked points. 
\end{defin}

If $\sigma: z \mapsto z+\epsilon$ and $(\alpha,\beta,\delta) \in D'_{g,l}$, then note that $\sigma(\alpha,\beta,\delta)=(\alpha,\beta+2\alpha \epsilon,\delta)$.
For any $(\alpha,\beta,\delta) \in D'_{g,0}$ we can find $\sigma \in G_{g+1}$ such that  $\sigma(\alpha,\beta,\delta)=(1,0,\delta')$ for some $\delta'$. The action of $G_{g+1}$ on the elements of the form $(1,0,\delta')$ in $D'_{g,0}$ is then the same as the action of $\mathrm{GL}_2(k)$ on square-free binary forms of degree $2g+2$. This stack quotient is well known to be isomorphic to $\mathcal H_{g,0}(k)$, which proves the following lemma. 

\begin{lem} \label{lem-Hg} For any $g \geq 2$, 
\[
\mathcal{H}_{\mathbb{F}_{g+1},g} \cong \mathcal H_{g,0}(k).
\]
\end{lem}

\begin{rk} For any $g \geq 2$, we know that $\mathcal{H}_{g,0}$ has the rational cohomology of a point. Note that Lemma~\ref{lem-Hg} does not contradict Theorem~\ref{mainTheo}, since in this case the only stable cohomology group is the one of degree zero.
\end{rk}

\begin{defin} \label{def-2} 
For $g \geq 2$ and $i,j \geq 0$, let $\mathcal H_{g,i,j} \subset \mathcal H_{g,i+j}$ denote the moduli space (stack) of smooth hyperelliptic curves of genus $g$ together with $i$ marked ramification points and $j$ marked non-ramification points for which no two of the points are interchanged by the hyperelliptic involution. There is an action of $\mathfrak{S}_i \times \mathfrak{S}_j$ on $\mathcal H_{g,i,j}$ by permuting the elements of the two tuples of marked points.
\end{defin}

 We see that $\mathcal H_{g,i,j}(k)/(\mathfrak{S}_i \times \mathfrak{S}_j)$ consists of isomorphism classes of hyperelliptic curves defined over $k$ of genus $g$ together with two unordered tuples of points, as above, but which are defined over $\bar k$ and are, as tuples, fixed by the action of the Galois group of $\bar k$ over $k$. 

\begin{defin} For $g \geq 2$, $l \geq 0$ such that $g \geq l$, let $D^0_{g,l}$ denote the subset of $D'_{g,l}$ consisting of all triples $(\alpha,\beta,\delta)$ such that $\alpha$ is square free, and put $\mathcal{D}^0_{g,l}:=[D^0_{g,l}/G_{g+1-l}]$. 
\end{defin}

\begin{lem} For any $g \geq 2$, $l \geq 0$ such that $g \geq l$,  
\begin{equation}\label{eq1}
    \mathcal{D}^0_{g,l} \cong \bigcup_{i+j=l} \mathcal H_{g,i,j}(k)/(\mathfrak S_i \times \mathfrak S_j).
\end{equation}
\end{lem}

\begin{proof}
Fix a square-free homogeneous polynomial $\delta$ of degree $2g+2$ and let $C_{\delta}$ denote the curve given by the equation $z^2-\delta=0$. Say that $\alpha=\prod_{i=1}^j \alpha_i$, where the $\alpha_i$ are distinct irreducible homogeneous polynomials of degree $d_i$. Fix $i$. At least one of $\alpha_i(x,1)$ and $\alpha_i(1,y)$ has degree $d_i$. We assume that the former holds. The only case we miss is when $\alpha_i=y$ and $d_i=1$ (the point at infinity), that can be treated analogously. Let $\pi_i$ denote $x$ in the degree $d_i$ field extension $k_i:=k[x]/\alpha_i(x,1)$ of $k$. Using this we will view the choice of $\alpha$ as a choice of an unordered $l$-tuple of distinct points defined over $\bar k$, fixed as a tuple by the Galois group of $\bar k$ over $k$, on the base of the hyperelliptic cover $\psi: C_{\delta} \to \mathbf P^1$, for which $\psi(x,y,z)=(x,y)$. 
The polynomial $\beta^2-\delta$, for a homogeneous polynomial $\beta$ of degree $g+1$, is divisible by $\alpha_i$ precisely if $\beta(\pi_i)^2=\delta(\pi_i)$. Or in other words precisely if $(\pi_i,1,\beta(\pi_i)) \in C_{\delta}$. The equality $\beta(\pi_i)^2=\delta(\pi_i)$, for all $i$, then determines $\beta$ modulo a multiple of $\alpha$. 

The isomorphisms between the pointed hyperelliptic curves of this form are precisely given by the elements of $G_{g+1-l}$. 
\end{proof}

Partitions of an integer $l$ will be written on the form $\lambda=[1^{\lambda_1},\ldots,l^{\lambda_l}]$. Recall that $|\lambda|=\sum_{r=1}^l r\lambda_r=l$ and that $l(\lambda)=\sum_{r=1}^l \lambda_i$. 

\begin{defin}\label{def-3} For any $g \geq 2$, $m \geq 0$ and partition $\lambda$, let $\mathcal H_{g,m,\lambda} \subset \mathcal H_{g,m,l(\lambda)}$ denote the subset consisting of curves with $m$ marked ramification points and tuples of $\lambda_i$ marked non-ramification points (for which no two of the points are interchanged by the hyperelliptic involution) for $i=1,\ldots,l$. The group $\mathfrak S_m \times \mathfrak S_{\lambda}$ with $\mathfrak S_{\lambda}:=\times_{r=1}^l \mathfrak S_{\lambda_i}$ acts on the marked points. 
 
\end{defin}

The following proposition gives a description of the moduli space of hyperelliptic curves embedded in a Hirzebruch surface in terms of moduli spaces of pointed (non-embedded) hyperelliptic curves.

\begin{prop} \label{prop-moduli} For any $g \geq 2$, $l \geq 0$ such that $g \geq l$, 
\begin{equation}\label{eq3}
    \mathcal{H}_{\mathbb{F}_{g+1-l},g} \cong \bigcup_{m=0}^l \bigcup_{\lambda \vdash l-m} \mathcal H_{g,m,\lambda}(k)/(\mathfrak S_m \times \mathfrak S_{\lambda}).
\end{equation}
\end{prop}
\begin{proof}

Fix any $m \geq 0$ and a partition $\lambda$ of $l-m$ and let $D_{g,m,\lambda}$ denote the subset of $D'_{g,l}$ consisting of all triples $(\alpha,\beta,\delta)$ such that $\alpha=\alpha' \, \alpha''$ where $\alpha'$ is square-free and divides $\delta$ and $\alpha''$ is coprime to $\delta$ and of type $\lambda$. Put $\mathcal{D}_{g,m,\lambda}:=[D_{g,m,\lambda}/G_{g+1-l}]$. From the definition we see that 
\[
\mathcal{D}'_{g,l} \cong \bigcup_{m=0}^l \bigcup_{\lambda \vdash l-m} \mathcal{D}_{g,m,\lambda}.
\]
We will show that 
\begin{equation}\label{eq2}
    \mathcal{D}_{g,m,\lambda} \cong \mathcal H_{g,m,l(\lambda)}(k)/(\mathfrak S_m \times \mathfrak S_{|\lambda|}),
\end{equation}
from which the proposition follows. 

Say that $\alpha=\prod_{i=1}^l \prod_{j=1}^{r_i} \alpha_{ij}^i$, for some distinct homogeneous irreducible polynomials $\alpha_{ij}$ of degree $d_{ij}$ such that $\prod_{j=1}^{r_i} \alpha_{ij}$ has degree $\lambda_i$. We will say that $\alpha$ has type $\lambda$ if it is of this form. Fix again a square-free homogeneous polynomial $\delta$ and let $\pi_{ij}$ be a root of $\alpha_{ij}(x,1)$ as above. Write a polynomial $\beta(x,1)$ on the form 
\[
\beta(x,1)=\alpha_{ij}(x,1)^i\, \tilde \beta(x)+\sum_{r=0}^{i-1} \alpha_{ij}(x,1)^r \, \beta_r(x)  
\]
for some polynomial $\tilde \beta(x)$ and polynomials $\beta_r(x)$ of degree $d_{ij}-1$ for all $0 \leq r \leq i-1$. Fix $\beta_0(x)$ such that $\beta_0(\pi_{ij})^2=\delta(\pi_{ij},1)$. We then see that the equation $\beta(x,1)^2=\delta(x,1)$ modulo $\alpha_{ij}^i$ has a unique solution as long as $\delta(\pi_{ij},1) \neq 0$. If $\delta(\pi_{ij},1)=0$ then it follows that $i=1$ because $\delta$ is square-free. 
\end{proof}

\begin{rk} The construction above can be made in a more geometric way. The polynomial $\alpha$ is then viewed as a subscheme of length $l=(2E_{g+1-l}+(2g+2-l)F_{g+1-l})\cdot E_{g+1-l}$
of the intersection between the hyperelliptic curve and the exceptional curve $E_{g+1-l}$.
With this view, the moduli space $\mathcal{D}_{g,l}$ parametrizes hyperelliptic curves $C$ lying on $\mathbb{F}_{g+1-l}$ meeting the exceptional curve in a subscheme of length $l$.

Let us sketch the geometric argument. Recall first that one can obtain the $n+1$-th Hirzebruch surface from $\mathbb{F}_n$ by blowing up any point $p$ on the exceptional curve $E_n$ and then by contracting the proper transform of the line of the ruling through $p$, \cite[4.3]{GH}.
We apply this transformation, which is birational, to any point $p$ of intersection between $C$ and $E_{g+1-l}$. If $p$ has multiplicity one, then locally this transformation is described in Figure \ref{figure1}. We thus obtain a smooth curve $\tilde{C}$ which is still hyperelliptic and birational, hence isomorphic, to $C$.

\begin{figure}[H]
\centering
\includegraphics[width=\textwidth]{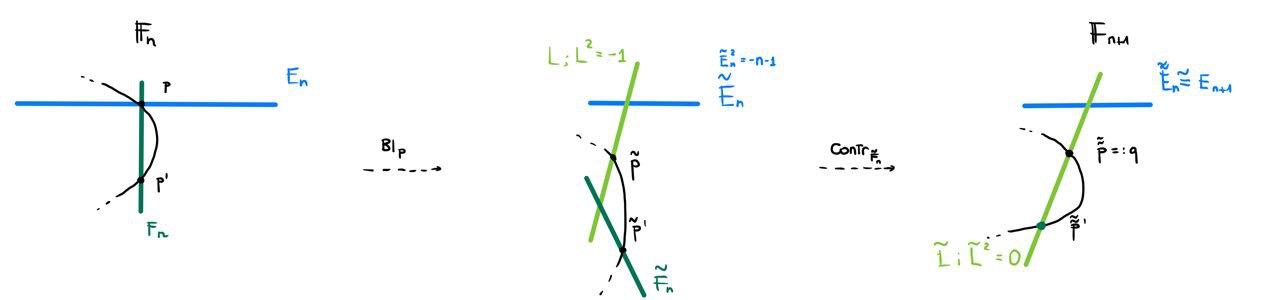}
\caption{Construction from $\mathbb{F}_{n}$ to $\mathbb{F}_{n+1}$ if $p$ meets $E_n$ transversally.}\label{figure1} 
\end{figure}

Notice that the above transformation can also be reversed, but in order to do so we need to mark the point $q$ defined as the image of $p$. If all points of the intersection are transversal, corresponding to $\mathcal{D}^0_{g,l}$, then we can repeat the procedure $l$ times, for each point of intersection, and obtain a birational morphism $\mathbb{F}_{g+1-l}\rightarrow\mathbb{F}_{g+1}$
which restricts to an isomorphism on $C$. 

If any point of intersection has multiplicity $\nu>1$, it suffices to repeat the transformation $\nu-1$ times to get to the previous setting, that is, to obtain a curve meeting the exceptional curve transversally. For instance, if $\nu=2$, then locally this transformation is described in Figure \ref{figure2}. 
This gives a curve as in the previous setting and we can repeat the above transformation.

\begin{figure}[H]\centering
 \includegraphics[width=\textwidth]{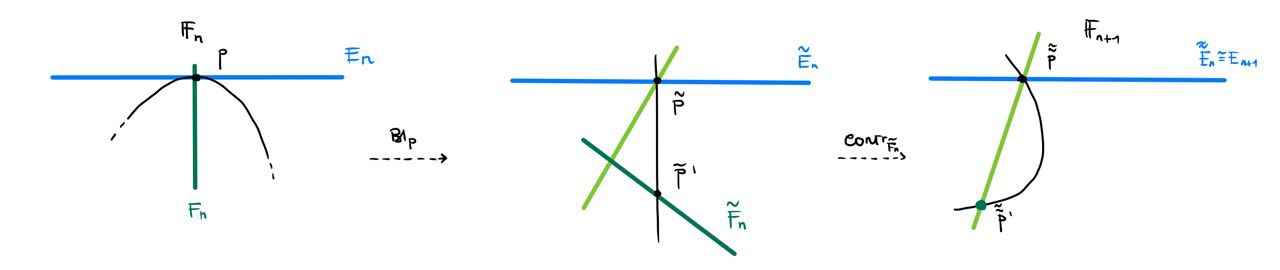}
\caption{Construction from $\mathbb{F}_{n}$ to $\mathbb{F}_{n+1}$ if $p$ meets $E_n$ with multiplicity $2$.}\label{figure2}
\end{figure}
\end{rk}

\begin{rk} \label{rk-H} Let $G'_0$ denote the subgroup of $G_0$ for which $x_1 \mapsto x_1$. We can then go through the same arguments as above for $g \geq 2$ and $l=g+1$, with the corresponding definitions of $D_{g,g+1}$, $D_{g,g+1}$ etc, but where we consider the action of $G_0'$ rather than that of $G_0$. Putting $\mathcal{H}'_{\mathbb{F}_0,g}:=\left[D_{g,g+1}/G'_0\right]$ we for instance have the corresponding statement to Proposition~\ref{prop-moduli}, 
that 
\begin{equation}\label{eq4}
    \mathcal{H}'_{\mathbb{F}_0,g} \cong \bigcup_{m=0}^{g+1} \bigcup_{\lambda \vdash g+1-m} \mathcal H_{g,m,\lambda}(k)/(\mathfrak S_m \times \mathfrak S_{\lambda}).
\end{equation}
\end{rk}
\end{subsection}

\begin{subsection}{Counting points over finite fields} \label{sec-lefschetz}
In this final section we make point counts over finite fields of $\mathcal{H}_{\mathbb{F}_{g+1-l},g}$ for $g \geq 2$, $l \geq 0$ such that $g \geq l$. These give independent results as well as consistency checks with Theorem~\ref{mainTheo}, using that the Grothendiek-Lefschetz fixed points theorem connects point counts over finite fields to the cohomology of $\mathcal{H}_{\mathbb{F}_{g+1-l},g}$. 

Using the notation and results of \cite{Bhyp} (in particular the results of Section~12)  together with Proposition~\ref{prop-moduli} we can find an expression for the number of points of $\mathcal{H}_{\mathbb{F}_{g+1-l},g}$, for $g \geq 2$, $l \geq 0$ such that $g \geq l$, over any finite field $k=\mathbf F_q$, with $q$ odd, and $q \equiv_n 1$ if $n\geq 3$, in terms of $u_g$'s of degree at most $l$. In \cite[Section 7]{Bhyp} all $u_g$'s of degree at most $5$ are determined. Using information on $\mathcal M_{1,1}[2]$, the moduli space of elliptic curves together with a level $2$ structure, we can also determine all $u_g$'s of degree $6$. The results for $1 \leq l \leq 3$ are given by, 
\[
\# \mathcal{H}_{\mathbb{F}_{g},g}(\mathbf F_q)=(q+1)q^{2g-1}, \;
\# \mathcal{H}_{\mathbb{F}_{g-1},g}(\mathbf F_q)=(q+1)q^{2g}-\delta^2_g, \;
\# \mathcal{H}_{\mathbb{F}_{g-2},g}(\mathbf F_q)=(q+1)(q^{2g+1}-\delta^2_g),
\]
where $\delta^2_g$ is $1$ if $g$ is odd, and $0$ if $g$ is even. Replacing $q$ by $\mathbf L$ in the formulas for $\# \mathcal{H}_{\mathbb{F}_{g-i},g}$ with $i=0,1,2$ above, gives the Hodge-Grothendieck (compactly supported) Euler characteristic of $\mathcal{H}_{\mathbb{F}_{g-i},g}$ over $\mathbf C$, see \cite[Theorem 2.1.8]{HRV}. 

Following \cite[Section 13]{Bhyp} we can divide $\# \mathcal{H}_{\mathbb{F}_{g+1-l},g}(\mathbf F_q)$ into a so-called stable part $\#^{\mathrm{st}} \mathcal{H}_{\mathbb{F}_{g+1-l},g}(\mathbf F_q)$, which we can determine for any $g \geq 2$ and $l \geq 0$ such that $g \geq l$, and an unstable part $\#^{\mathrm{unst}} \mathcal{H}_{\mathbb{F}_{g+1-l},g}(\mathbf F_q)$. 
For two polynomials $f_1$ and $f_2$ in $q$, let $[f_1/f_2]$ denote the polynomial quotient. The stable part for $l=4$ equals 
\[
\#^{\mathrm{st}} \mathcal{H}_{\mathbb{F}_{g-3},g}(\mathbf F_q)= \left[\frac{q^{2g}(q^6+2q^5+2q^4+2q^3+q^2+1)}{(q^2+1)(q+1)} \right]. 
\]
The unstable part, for $l=4$, is a polynomial in $q$ of degree $3$ whose coefficients depend upon $g$ modulo $12$. For instance, we have that  
\[
\#^{\mathrm{unst}} \mathcal{H}_{\mathbb{F}_{g-3},g}(\mathbf F_q)=(-3g^2-g)q^3+(6g^2-g-1)q^2+(3g^2-4g-1)q-6g^2+5g
\]
for all $g \geq 2$ that are $0$ modulo $12$. 

Let us put $Q_{g,l}(q):=\#^{\mathrm{st}} \mathcal{H}_{\mathbb{F}_{g+1-l},g}(\mathbf F_q)$, which will be a polynomial in $q$ of degree $2g-1+l$, for each $g\geq 2$ and $l\geq 0$ such that $g \geq l$. 
Our guess, starting from \cite[Theorem 2.1.8]{HRV}, would be that we can replace $q$ by $\mathbf L$ in the polynomial $Q_{g,l}(q)$ to get the stable cohomology of $\mathcal{H}_{\mathbb{F}_{g+1-l},g}$ over $\mathbf C$. 
Indeed we have checked (note that one has to dualize to go from usual to compactly supported cohomology) that, for all $g \geq 2$ and $1 \leq l \leq 15$ such that $g \geq l$, 
\[
(1+\mathbf L)P^{\mathrm{st}}(-1)=
\mathbf{L}^{2g-1+l} \cdot Q_{g,l}(\mathbf{L}^{-1}) \;\; \text{up to degree } \lceil 3l/2 \rceil.
\]

The same result holds also for $Q_{g,g+1}(q):=\#^{\mathrm{st}} \mathcal{H}'_{\mathbb{F}_0,g}(\mathbf F_q)$. 

Note also that 
\[
\# \mathcal{H}'_{\mathbb{F}_0,3}(\mathbf F_q)=q^8(q+1), \; \# \mathcal{H}'_{\mathbb{F}_0,4}(\mathbf F_q)=(q+1) q^2 (q^9+q^3-q^2-q-1),
\]
which both are visibly divisible by $q+1$. This divisibility follows 
from the relation between $\mathcal{H}_{\mathbb{F}_0,g}$ and $\mathcal{H}'_{\mathbb{F}_0,g}$, see Remark~\ref{rk-H}. 
\end{subsection}

\section*{Acknowledgements}
The second author is grateful to Stockholm University for the hospitality during her visiting period in the Fall of 2022 supported by Fondazione Ing. Aldo Gini. The second author is a member of the INDAM group GNSAGA.
Both authors would like to thank Orsola Tommasi for proposing the project and for her precious help and suggestions for the initial ideas in Sections~\ref{sec5_1} and \ref{sec-config2} and \ref{sec-iso}. We would also like to thank her for her comments upon an earlier version of the article.
\bibliographystyle{alpha}
\bibliography{bibliography}

\end{document}